\documentclass[11pt,reqno]{amsart} 
\usepackage[applemac]{inputenc}
\usepackage{amsmath}
\usepackage{amssymb}
\usepackage{euscript}
\usepackage{mathrsfs}
\usepackage{wasysym}
\usepackage[all]{xy}
\usepackage{graphicx}
\usepackage{color}
\usepackage{multirow}
\usepackage{pifont}
\usepackage[table]{xcolor}
\usepackage[colorlinks=true,linktocpage=true,pagebackref=false, urlcolor=black, citecolor=black,linkcolor=black]{hyperref}
\usepackage{tikz}
\usepackage{caption,subcaption}

\newtheorem{thm}{Theorem}[section]

\newtheorem{lem}[thm]{Lemma}
\newtheorem{prop}[thm]{Proposition}

\newtheorem{cor}[thm]{Corollary}

\theoremstyle{definition}
\newtheorem{defn}[thm]{Definition}
\newtheorem{defns}[thm]{Definitions}

\theoremstyle{remark}
\newtheorem{remark}[thm]{Remark}
\newtheorem{remarks}[thm]{Remarks}
\newtheorem{example}[thm]{Example}

\numberwithin{equation}{section}

\newcommand{\K}{{\mathbb K}} 
\newcommand{\Z}{{\mathbb Z}} \newcommand{\R}{{\mathbb R}}
 \newcommand{\C}{{\mathbb C}}

\newcommand{\sph}{{\mathbb S}}

\newcommand{\an}{{\mathcal O}} 
 
 \newcommand{\J}{{\mathcal J}}
 
 \newcommand{\I}{{\mathcal I}}
\newcommand{\ideal}{{\mathcal I}}
\newcommand{\ceros}{{\mathcal Z}}

\newcommand{\gtp}{{\mathfrak p}} \newcommand{\gtq}{{\mathfrak q}}
\newcommand{\gtm}{{\mathfrak m}} \newcommand{\gtn}{{\mathfrak n}}
\newcommand{\gta}{{\mathfrak a}} \newcommand{\gtb}{{\mathfrak b}}

 \newcommand{\gtA}{{\mathfrak A}}

\newcommand{\Jhaz}{{\EuScript I}}

\newcommand{\Ss}{{\EuScript S}}

\newcommand{\Reg}{\operatorname{Reg}}
\newcommand{\Sing}{\operatorname{Sing}}

\newcommand{\Int}{\operatorname{Int}}
\newcommand{\cl}{\operatorname{Cl}}

\newcommand{\supp}{\operatorname{supp}}

\newcommand{\id}{\operatorname{id}}

\newcommand{\zar}{\operatorname{zar}}

\newcommand{\veps}{\varepsilon}

\newcommand{\ol}{\overline}

\numberwithin{equation}{section}

\begin{document}
\title[On the irreducible components of globally defined semianalytic sets]{On the irreducible components\\ of globally defined semianalytic sets}

\author{Jos\'e F. Fernando}
\address{Departamento de \'Algebra, Facultad de Ciencias Matem\'aticas, Universidad Complutense de Madrid, 28040 MADRID (SPAIN)}
\email{josefer@mat.ucm.es}

\date{21/04/2015}

\thanks{Author supported by Spanish GAAR MTM2011-22435, Spanish MTM2014-55565, Grupos UCM 910444 and the ``National Group for Algebraic and Geometric Structures, and their Applications'' (GNSAGA - INdAM). This article has been mainly written during a one-year research stay of the author in the Dipartimento di Matematica of the Universit\`a di Pisa. The author would like to thank the department for the invitation and the very pleasant working conditions. The one-year research stay of the author is partially supported by MECD grant PRX14/00016.}

\begin{abstract}--- \ 
In this work we present the concept of \em amenable $C$-semianalytic subset \em of a real analytic manifold $M$ and study the main properties of this type of sets. Amenable $C$-semianalytic sets can be understood as globally defined semianalytic sets with a neat behavior with respect to Zariski closure. This fact allows us to develop a natural definition of \em irreducibility \em and the corresponding \em theory of irreducible components \em for amenable $C$-semianalytic sets. These concepts generalize the parallel ones for: complex algebraic and analytic sets, $C$-analytic sets, Nash sets and semialgebraic sets. 
\end{abstract}
\maketitle

\section{Introduction}\label{s1}

Irreducibility and irreducible components are usual concepts in Geometry and Algebra. Both concepts are strongly related with prime ideals and primary decomposition of ideals. There is an important background concerning this matter in Algebraic and Analytic Geometry. These concepts has been satisfactorily developed for complex algebraic sets (Lasker-N\"other \cite{la}), complex analytic sets and Stein spaces (Cartan \cite{c1}, Forster \cite{of}, Remmert-Stein \cite{rs}), global real analytic sets introduced by Cartan (also known as $C$-analytic sets, Whitney-Bruhat \cite{wb}) and Nash sets (Efroymson \cite{e}, Mostowski \cite{m}, Risler \cite{r}). 

Recall that a subset $X$ of a real analytic manifold $M$ is \em (real) analytic \em if for each point $x\in M$, there exists an open neighborhood $U^x$ such that $X\cap U^x=\{f_1=0,\ldots,f_r=0\}\subset U^x$ for some $f_1,\ldots,f_r\in\an(U^x)$. The global behavior of real analytic sets could be wild as it is shown in the exotic examples presented in \cite{bc,c,wb} and this blocks the possibility of having a reasonable concept of irreducibility. A subset $X\subset M$ is a \em $C$-analytic set \em if $X=\{f_1=0,\ldots,f_r=0\}$ for some $f_1,\ldots,f_r\in\an(M)$. The difference with analytic sets focuses on the global character of the functions defining $X$. As we have already mentioned, $C$-analytic sets have a good global behavior that enables a consistent concept of irreducibility.

In Real Geometry also appear naturally sets described by inequalities. Recall that a subset $S\subset\R^n$ is \em semialgebraic \em if it has a description as a finite boolean combination of polynomial equalities and inequalities. In \cite{fg} we presented consistent concepts of irreducibility and irreducible components and it is natural to wonder it they extend to the semianalytic setting.

A subset $S$ of a real analytic manifold $M$ is \em semianalytic \em if for each point $x\in M$ there exists an open neighborhood $U^x$ such that $S\cap U^x$ is a finite union of sets of the type $\{f=0,g_1>0,\ldots,g_r>0\}\subset U^x$ where $f,g_i\in\an(U^x)$ are analytic functions on $U^x$. This class, that includes analytic sets, is too large to afford the concepts of irreducibility and irreducible components and we must reduce it in order to have any chance of success. As it happens with $C$-analytic sets, some `global restriction' should be required. In \cite{abf1} we have recently introduce a class of globally defined semianalytic sets.

A subset $S$ of a real analytic manifold $M$ is a \em $C$-semianalytic set \em if for each point $x\in M$, there exists an open neighborhood $U^x$ such that $S\cap U^x$ is a finite boolean combination of equalities and inequalities on $M$, that is, $S\cap U^x$ is a finite union of sets of the type $\{f=0,g_1>0,\ldots,g_r>0\}$ where $f,g_i\in\an(M)$ are global analytic functions. The difference with classical semianalytic sets concentrates on the fact that the analytic functions defining $S$ in a small neighborhood of each point of $M$ are global. A particular relevant family of $C$-semianalytic sets is that of global $C$-semianalytic sets \cite{ac,rz1,rz2,rz3}. A {\em global $C$-semianalytic subset $S$ of $M$} is a finite union of {\em basic $C$-semianalytic sets}, that is, a finite union of sets of the type $\{f=0,g_1>0,\ldots,g_r>0\}$ where $f,g_j\in\an(M)$. One deduces using paracompactness of $M$ that $C$-semianalytic sets coincide with locally finite unions of global $C$-semianalytic sets \cite[Lem.3.1]{abf1}. We refer the reader to \cite{abf1} for a careful study of $C$-semianalytic sets. 

A second requirement to avoid pathologies should be that `Zariski closure preserve dimensions'. The \em Zariski closure \em of a subset $E\subset M$ is the smallest $C$-analytic subset $X$ of $M$ that contains $E$. We define the dimension of a $C$-semianalytic set $S\subset M$ as $\dim(S):=\sup_{x\in M}\{\dim(S_x)\}$ and refer the reader to \cite[VIII.2.11]{abr} for the dimension of semianalytic germs. The Zariski closure of a $C$-semianalytic set is in general a $C$-analytic set of higher dimension (see Example \ref{ojo}). To guarantee a satisfactory behavior of Zariski closure we need an even more restrictive concept. A subset $S\subset M$ is a \em amenable $C$-semianalytic set \em if it is a finite union of $C$-semianalytic sets of the type $X\cap U$ where $X\subset M$ is a $C$-analytic set and $U\subset M$ is an open $C$-semianalytic set. In particular the Zariski closure of $S$ has the same dimension as $S$. 

In this work we prove that amenable $C$-semianalytic sets admit a solid concept of \em irreducibility \em and a theory of \em irreducible components\em. Both generalize the parallel ones for: complex algebraic and analytic sets, $C$-analytic sets, Nash sets and semialgebraic sets mentioned above, that is, irreducibility or irreducible components of a set $S$ of one of the previous types coincides with irreducibility or irreducible components when we understand $S$ as an amenable $C$-semianalytic set. 

\subsection{Main results}
We present next the main results of this work.

\subsubsection{Main properties of amenable $C$-semianalytic sets}

The family of amenable $C$-semianalytic sets is closed under the following operations: finite unions and intersections, interior, connected components, sets of points of pure dimension $k$ and inverse images of analytic maps. However, it is not closed under: complement, closure, locally finite unions and sets of points of dimension $k$ (see Examples \ref{ex} and \ref{dimktamenot}). In addition a $C$-semianalytic set $S\subset M$ is amenable if and only if it is a locally finite contable union of basic $C$-semianalytic sets $S_i$ such that the family $\{\ol{S_i}^{\zar}\}_{i\geq1}$ of their Zariski closures is locally finite (after eliminating repetitions). As a consequence we show in Corollary \ref{lfuok2} that the union of a locally finite collection of amenable $C$-semianalytic sets whose Zariski closures constitute a locally finite family (after eliminating repetitions) is an amenable $C$-semianalytic set. 

\subsubsection{Images of amenable $C$-semianalytic sets under proper holomorphic maps}
Let $(X,\an_X)$ and $(Y,\an_Y)$ be reduced Stein spaces. Let $\sigma:X\to X$ and $\tau:Y\to Y$ be anti-involutions. Assume
$$
X^\sigma:=\{x\in X:\ x=\sigma(x)\}\quad\text{and}\quad Y^\tau:=\{y\in Y:\ y=\tau(y)\}
$$ 
are not empty sets. It holds that $(X^\sigma,\an_{X^\sigma})$ and $(Y^\tau,\an_{Y^\tau})$ are real analytic spaces. Observe that $(X,\an_X)$ and $(Y,\an_Y)$ are complexifications of $(X^\sigma,\an_{X^\sigma})$ and $(Y^\tau,\an_{Y^\tau})$. We say that a $C$-semianalytic set $S\subset X^\sigma$ is \em ${\mathcal A}(X^\sigma)$-definable \em if for each $x\in X^\sigma$ there exists an open neighborhood $U^x$ such that $S\cap U^x$ is a finite union of sets of the type $\{F|_{X^\sigma}=0,G_1|_{X^\sigma}>0,\ldots,G_r|_{X^\sigma}>0\}$ where $F,G_i\in\an(X)$ are invariant holomorphic sections. We denote the set of $\sigma$-invariant homolomorphic functions of $X$ restricted to $X^\sigma$ with ${\mathcal A}(X^\sigma)$.

\begin{thm}\label{properint-tame}
Let $F:(X,\an_X)\to(Y,\an_Y)$ be an invariant proper holomorphic map, that is, $\tau\circ F=F\circ\sigma$. Let $S\subset X^\sigma$ be a ${\mathcal A}(X^\sigma)$-definable and amenable $C$-semianalytic set and let $S'\subset Y^\tau$ be an amenable $C$-semianalytic set. We have 
\begin{itemize}
\item[(i)] $F(S)$ is an amenable $C$-semianalytic subset of $Y^\tau$ of the same dimension as $S$.
\item[(ii)] If $T$ is a union of connected components of $F^{-1}(S')\cap X^\sigma$, then $F(T)$ is an amenable $C$-semianalytic set.
\end{itemize}
\end{thm}

\subsubsection{Tameness algorithm}
There exist many $C$-semianalytic sets that are not amenable, see Examples \ref{tamenot} and \ref{tamenot0}. In addition the family of amenable $C$-semianalytic sets is wider than the family of global $C$-semianalytic sets, see Example \ref{buttame}. In \ref{algorithmtame} we develop an algorithm, involving topological and algebraic operations, to determine if a $C$-semianalytic set is amenable.

\subsubsection{Irreducibility}
In the second part of the article we analyze the concept of irreducibility for amenable $C$-semianalytic sets and develop a theory of irreducible components that generalize the classical ones for other families of classical sets that admit a structure of amenable $C$-semianalytic set (complex algebraic and analytic sets, $C$-analytic sets, Nash sets, semialgebraic sets, etc.).
{\small
\begin{figure}[ht]
\centerline{\xymatrix{
*+[F]{\txt{Nash\\ sets}}\ar@{=>}[r]\ar@{=>}[rrd]&*+[F]{\txt{Semialgebraic\\ sets}}\ar@{=>}[r]&*+[F]{\txt{Global $C$-semi-\\analytic sets}}\ar@{=>}[r]&*+[F]{\txt{Tame $C$-semi-\\analytic sets}}\ar@{=>}[r]&*+[F]{\txt{$C$-semianalytic\\ sets}}\ar@{=>}[d]\\
*+[F]{\txt{Complex\\ algebraic sets}}\ar@{=>}[r]&*+[F]{\txt{Complex\\ analytic sets}}\ar@{:>}_{X^\R}[r]&*+[F]{\txt{$C$-analytic\\ sets}}\ar@{=>}[r]\ar@{=>}[u]&*+[F]{\txt{Analytic\\sets}}\ar@{=>}[r]&*+[F]{\txt{Semianalytic\\sets}}
}}
\caption{Relations between the different types of sets appearing in this work}
\end{figure}
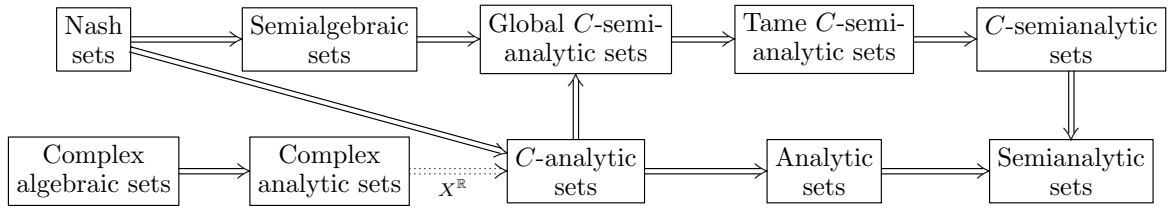
}

In the algebraic, complex analytic, $C$-analytic and Nash settings a geometric object is irreducible if it is not the union of two proper geometric objects of the same nature. In the amenable $C$-semianalytic setting this definition does not work because every $C$-semianalytic set with at least two points would be reducible. Indeed, if $p,q\in S$ and $W$ is open $C$-semianalytic neighborhood of $p$ in $M$ such that $q\not\in W$, it holds
$$
S=(S\cap W)\cup(S\setminus\{p\})
$$
where $S\cap W$ and $S\setminus\{p\}$ are amenable $C$-semianalytic sets.

In the previous settings the irreducibility of a geometric object $X$ is equivalent to the fact that the corresponding ring of polynomial, analytic or Nash functions on $X$ is an integral domain. This equivalence suggests us to attach to each amenable $C$-semianalytic set $S\subset M$ the ring $\an(S)$ of real valued functions on $S$ that admit an analytic extension to an open neighborhood of $S$ in $M$. We say that $S$ is \em irreducible \em if and only if $\an(S)$ is an integral domain.

Our definition extends the notion of irreducibility for $C$-analytic, semialgebraic and Nash sets. In addition if $X\subset\C^n$ is complex analytic set and $X^\R\subset\R^{2n}$ is its underlying real analytic structure, \em $X$ is irreducible as a complex analytic set if and only if $X^\R$ is irreducible as a $C$-semianalytic set\em.

\subsubsection{Irreducibility vs connectedness} The irreducibility of an amenable $C$-semianalytic set $S$ has a close relation with the connectedness of certain subset of the normalization of the Zariski closure of $S$. Let $S\subset M$ be an amenable $C$-semianalytic set and let $X$ be its Zariski closure. Let $(\widetilde{X},\sigma)$ be a Stein complexification of $X$ together with the anti-involution $\sigma:\widetilde{X}\to\widetilde{X}$ whose set of fixed points is $X$. Let $(Y,\pi)$ be the normalization of $\widetilde{X}$ and let $\widehat{\sigma}:Y\to Y$ be the anti-holomorphic involution induced by $\sigma$ in $Y$, which satisifies $\pi\circ\widehat{\sigma}=\sigma\circ\pi$.

\begin{thm}\label{dpm}
The amenable $C$-semianalytic set $S$ is irreducible if and only if there exists a connected component $T$ of $\pi^{-1}(S)$ such that $\pi(T)=S$.
\end{thm}

\subsubsection{Irreducible components} 
We develop a satisfactory theory of irreducible components for amenable $C$-semianalytic sets. As before, it holds that if $S$ is either $C$-analytic, semialgebraic or Nash, then its irreducible components as a set of the corresponding type coincide with the irreducible components of $S$ as an amenable $C$-semianalytic set. In addition if $X\subset\C^n$ is complex analytic set and $X^\R\subset\R^{2n}$ is its underlying real analytic structure, \em the underlying real analytic structures of the irreducible components of $X$ as a complex analytic set coincide with the irreducible components of $X^\R$ as a $C$-semianalytic set\em.

\begin{defn}[Irreducible components]\label{irredcomptame}
Let $S\subset M$ be an amenable $C$-semi\-analytic set. A countable locally finite family $\{S_i\}_{i\geq1}$ of amenable $C$-semianalytic sets that are contained in $S$ is \em a family of irreducible components of \em $S$ if the following conditions are fulfilled: 
\begin{itemize}
\item[(1)] Each $S_i$ is irreducible.
\item[(2)] If $S_i\subset T\subset S$ is an irreducible amenable $C$-semianalytic set, then $S_i=T$.
\item[(3)] $S_i\neq S_j$ if $i\neq j$.
\item[(4)] $S=\bigcup_{i\geq1} S_i$.
\end{itemize} 
\end{defn}

In Theorem \ref{irredcomp3} we prove the existence and uniqueness of the family of irreducible components of an arbitrary amenable $C$-semianalytic set $S$. In particular, we show the following.

\begin{thm}\label{bijection}
There exists a bijection between the irreducible components of an amenable $C$-semi\-analytic set $S\subset M$ and the minimal prime ideals of the ring $\an(S)$.
\end{thm}

The behavior of the Zariski closure of an amenable $C$-semianalytic set $S$ in a small enough open neighborhood $U\subset M$ of $S$ is neat with respect to the irreducible components.

\begin{prop}\label{neatirred}
Let $S\subset M$ be an amenable $C$-semi\-analytic set. There exists an open neighborhood $U\subset M$ of $S$ such that if $X$ is the Zariski closure of $S$ in $U$ and $\{X_i\}_{i\geq1}$ are the irreducible components of $X$, then $\{S_i:=X_i\cap S\}_{i\geq1}$ is the family of the irreducible components of $S$ and $X_i$ is the Zariski closure of $S_i$ in $U$ for $i\geq1$. In particular, if $S$ is a global $C$-semianalytic subset of $M$, each $S_i$ is a global $C$-semianalytic subset of $U$.
\end{prop}

The family $\{\ol{S_i}^{\zar}\}_{i\geq1}$ of the Zariski closures of the irreducible components $\{S_i\}_{i\geq1}$ of an amenable $C$-semianalytic set is by Proposition \ref{zclc} locally finite in $M$ (after eliminating repetitions). Consequently,

\begin{cor}
Any union of irreducible components of an amenable $C$-semianalytic $S\subset M$ is an amenable $C$-analytic set.
\end{cor}

\subsection{Structure of the article} 
 
The article is organized as follows. In Section \ref{s2} we present all basic notions and notations used in this article as well as some preliminary results concerning the ring of analytic functions on a $C$-semianalytic set. We also include an Appendix \ref{app} including some results about locally finite families. The reading can be started directly in Section \ref{s3} and referred to the Preliminaries and the Appendix only when needed. In Section \ref{s3} we study the main properties of amenable $C$-semianalytic sets, we prove Theorem \ref{properint-tame} and we present an algorithm to determine the amenablility of a $C$-semianalytic set. The purpose of Section \ref{s4} is to understand the concept of irreducibility for amenable $C$-semianalytic sets while in Section \ref{s5} we prove the existence an uniqueness of the family of the irreducible components of an amenable $C$-semianalytic set. This requires the intermediate concept of \em weak irreducible components\em. We also study in Section \ref{s5} the relationship between the irreducible components of an amenable $C$-semianalytic set and the connected components of its inverse image under the normalization map.


\section{Preliminaries on real and complex analytic spaces}\label{s2}

Although we deal with real analytic functions, we make an extended use of complex analysis. In the following \em holomorphic \em refer to the complex case and \em analytic \em to the real case. For a further reading about complex analytic spaces we suggest \cite{gr} while we remit the reader to \cite{gmt,t} for the theory of real analytic spaces. The elements of $\an(X):=H^0(X,\an_X)$ are denoted with capital letters if $(X,\an_X)$ is a Stein space and with small letters if $(X,\an_X)$ is a real analytic space. All concepts appearing in the Introduction that involve a real analytic manifold $M$ can be extended to a real analytic space $(X,\an_X)$ using the ring $\an(X)$ of global analytic sections on $X$ instead of the ring $\an(M)$ of global analytic functions on $M$. 

\subsection{General terminology}
Denote the coordinates in $\C^n$ with $z:=(z_1,\ldots,z_n)$ where $z_i:=x_i+\sqrt{-1}y_i$. Consider the conjugation $\ol{\,\cdot\,}:\C^n\to\C^n,\ z\mapsto\ol{z}:=(\ol{z_1},\ldots,\ol{z_n})$ of $\C^n$, whose set of fixed points is $\R^n$. A subset $A\subset\C^n$ is \em invariant \em if $\ol{A}=A$. Obviously, $A\cap\ol{A}$ is the biggest invariant subset of $A$. Let $\Omega\subset\C^n$ be an invariant open set and $F:\Omega\to\C$ a holomorphic function. We say that $F$ is \em invariant \em if $F(z)=\ol{F(\ol{z})}$ for all $z\in\Omega$. This implies that $F$ restricts to a real analytic function on $\Omega\cap\R^n$. Conversely, if $f$ is analytic on $\R^n$, it can be extended to an invariant holomorphic function $F$ on some invariant open neighborhood $\Omega$ of $\R^n$. 

\subsubsection{Real and imaginary parts}

Write the tuple $z\in\C^n$ as $z=x+\sqrt{-1}y$ where $x:=(x_1,\ldots,x_n)$ and $y:=(y_1,\ldots,y_n)$, so we identify $\C^n$ with $\R^{2n}$. If $F:\Omega\to\C$ is a holomorphic function,
$$
F(x+\sqrt{-1}y):=\Re^*(F)(x,y)+\sqrt{-1}\Im^*(F)(x,y)
$$
where 
$$
\Re^*(F)(x,y):=\frac{F(z)+\ol{F(z)}}{2}\quad\text{and}\quad\Im^*(F)(x,y):=\frac{F(z)-\ol{F(z)}}{2\sqrt{-1}}
$$
are real analytic functions on $\Omega\equiv\Omega^\R$ understood as an open subset of $\R^{2n}$. 

Assume in addition that $\Omega$ is invariant. Then
$$
\Re(F):\Omega\to\C,\ z\mapsto\tfrac{F(z)+\ol{F(\ol{z})}}{2}\quad\text{and}\quad\Im(F):\Omega\to\C,\ z\mapsto\tfrac{F(z)-\ol{F(\ol{z})}}{2\sqrt{-1}}
$$ 
are invariant holomorphic functions that satisfy $F=\Re(F)+\sqrt{-1}\,\Im(F)$. We have
$$
\Re^*(F)=\Re^*(\Re(F))-\Im^*(\Im(F))\quad\text{and}\quad \Im^*(F)=\Im^*(\Re(F))+\Re^*(\Im(F)),
$$
so it is convenient not to confuse the pair of real analytic functions $(\Re^*(F),\Im^*(F))$ on $\Omega^\R$ with the pair of invariant holomorphic functions $(\Re(F),\Im(F))$ on $\Omega$.

\subsection{Reduced analytic spaces \cite[I.1]{gmt}}\label{ras}
Let $\K=\R$ or $\C$ and let $(X,\an_X)$ be an either complex or real analytic space. Let ${\mathcal F}_X$ be the sheaf of $\K$-valued functions on $X$ and let $\vartheta:\an_X\to{\mathcal F}_X$ be the morphism of sheaves defined for each open set $U\subset X$ by $\vartheta_U(s):U\to\K,\ x\mapsto s(x)$ where $s(x)$ is the class of $s$ module the maximal ideal $\gtm_{X,x}$ of $\an_{X,x}$. Recall that $(X,\an_X)$ is \em reduced \em if $\vartheta$ is injective. Denote the image of $\an_X$ under $\vartheta$ with $\an_X^r$ . The pair $(X,\an_X^r)$ is called the \em reduction \em of $(X,\an_X)$ and $(X,\an_X)$ is reduced if and only if $\an_X=\an_X^r$. The reduction is a covariant functor from the category of $\K$-analytic spaces to that of reduced $\K$-analytic spaces.

\subsection{Underlying real analytic space \cite[II.4]{gmt}}\label{uras}
Let $(Z,\an_Z)$ be a local model for complex analytic spaces defined by a coherent sheaf of ideals ${\mathcal I}\subset\an_{\C^n}|_{\Omega}$, that is, $Z:=\supp(\an_{\C^n}|_{\Omega}/{\mathcal I})$ and $\an_Z:=(\an_{\C^n}|_{\Omega}/{\mathcal I})|_Z$. Suppose that ${\mathcal I}$ is generated by finitely many holomorphic functions $F_1,\ldots,F_r$ on $\Omega$. Let ${\mathcal I}^\R$ be the coherent sheaf of ideals of $\an_{\R^{2n}}|_{\Omega^\R}$ generated by $\Re^*(F_i),\Im^*(F_i)$ for $i=1,\ldots,r$. Let $(Z^\R,\an_Z^\R)$ be the local model for a real analytic space defined by the coherent sheaf of ideals ${\mathcal I}^\R$. For every complex analytic space $(X,\an_X)$ there exists a structure of real analytic space on $X$ that we denote with $(X^\R,\an_X^\R)$ and it is called the \em underlying real analytic space of $(X,\an_X)$\em. The previous construction provides a covariant functor from the category of complex analytic spaces to that of real analytic spaces.

\begin{remark}
If $(X,\an_X)$ is a reduced complex analytic space, it may fail that $(X^\R,\an^\R_X)$ is coherent or reduced \cite[III.2.15]{gmt}. 
\end{remark}

\subsection{Anti-involution and complexifications}
Let $(X,\an_X)$ be a complex analytic space and let $(X^\R,\an_X^\R)$ be its underlying real analytic space. We regard both sheaves $\an_X$ of holomorphic sections and $\ol{\an}_X$ of antiholomorphic sections on $X$ as subsheaves of the sheaf $\an_X^\R\otimes\C$ (see \cite[III.4.5]{gmt}). An \em anti-involution \em on $(X,\an_X)$ is a morphism of $\R$-ringed spaces $\sigma:(X^\R,\an_X^\R\otimes\C)\to(X^\R,\an_X^\R\otimes\C)$ such that $\sigma^2=\id$ and it interchanges the subsheaf of holomorphic sections $\an_X$ with the subsheaf of antiholomorphic sections $\ol{\an}_X$. 

\subsubsection{Fixed part space}\label{fixed}
Let $(X,\an_X)$ be a complex analytic space endowed with an anti-involu\-tion $\sigma$. Let $X^\sigma:=\{x\in X:\ \sigma(x)=x\}$ and define a sheaf $\an_{X^\sigma}$ on $X^\sigma$ in the following way: for each open subset $U\subset X^\sigma$, we set $H^0(U,\an_{X^\sigma})$ as the subset of $H^0(U,\an_X|_{X^\sigma})$ of invariant sections. The $\R$-ringed space $(X^\sigma,\an_{X^\sigma})$ is called the \em fixed part space of $(X,\an_X)$ with respect to $\sigma$\em. By \cite[II.4.10]{gmt} it holds that $(X^\sigma,\an_{X^\sigma})$ is a real analytic space if $X^\sigma\neq\varnothing$ and it is a closed subspace of $(X^\R,\an_X^\R)$. By Cartan's Theorem B the natural homomorphism $H^0(X^\R,\an_X^\R)\to H^0(X^\sigma,\an_{X^\sigma})$ is surjective \cite[III.3.8]{gmt}. 

\subsubsection{Complexification and $C$-analytic spaces \cite[III.3]{gmt}}\label{complexification} 
A real analytic space $(X,\an_X)$ is a \em $C$-analytic space \em if it satisfies one of the following two equivalent conditions:
\begin{itemize}
\item[(1)] Each local model of $(X,\an_X)$ is defined by a coherent sheaf of ideals, which is not necessarily associated to a \em well reduced structure \em (see \ref{cas}).
\item[(2)] There exist a complex analytic space $(\widetilde{X},\an_{\widetilde{X}})$ endowed with an anti-holomorphic involution $\sigma$ whose fixed part space is $(X,\an_X)$.
\end{itemize}
The complex analytic space $(\widetilde{X},\an_{\widetilde{X}})$ is called a \em complexification \em of $X$ and it satisfies the following properties:
\begin{itemize}
\item[(i)] $\an_{\widetilde{X},x}=\an_{X,x}\otimes\C$ for all $x\in X$.
\item[(ii)] The germ of $(\widetilde{X},\an_{\widetilde{X}})$ at $X$ is unique up to an isomorphism.
\item[(iii)] $X$ has a fundamental system of invariant open Stein neighborhoods in $\widetilde{X}$. 
\item[(iv)] If $X$ is reduced, then $\widetilde{X}$ is also reduced.
\end{itemize} 
For further details see \cite{c,gmt,t,wb}. Observe that if $(X,\an_X)$ is a complex analytic space, $(X^\R,\an_X^\R)$ satisfies by definition condition (1) above, so it is a $C$-analytic space and has a well-defined complexification. 

\subsection{$C$-analytic sets}\label{cas} 
The concept of $C$-analytic sets was introduced by Cartan in \cite[\S7,\S10]{c}. Recall that a subset $X\subset M$ is \em $C$-analytic \em if there exists a finite set $\Ss:=\{f_1,\ldots,f_r\}$ of real analytic functions $f_i$ on $M$ such that $X$ is the common zero-set of $\Ss$. This property is equivalent to the following:
\begin{itemize}
\item[(1)] There exists a coherent sheaf of ideals $\Jhaz$ on $M$ such that $X$ is the zero set of $\Jhaz$.
\item[(2)] There exist an open neighborhood $\Omega$ of $M$ in a complexification $\widetilde{M}$ of $M$ and a complex analytic subset $Z$ of $\Omega$ such that $Z\cap M=X$.
\end{itemize}
A coherent analytic set is $C$-analytic. The converse is not true in general, consider for example Whitney's umbrella. 

\subsubsection{Well-reduced structure}\label{wrs}
Given a $C$-analytic set $X\subset M$ the largest coherent sheaf of ideals $\Jhaz$ having $X$ as its zero set is $\ideal(X)\an_M$ by Cartan's Theorem $A$, where $\ideal(X)$ is the set of all analytic functions on $M$ that are identically zero on $X$. The coherent sheaf $\an_X:=\an_M/\ideal(X)\an_M$ is called the \em well reduced structure of $X$\em. The $C$-analytic set $X$ endowed with its well reduced structure is a real analytic space, so it has a well-defined complexification as commented above. 

\subsubsection{Singular set of a $C$-analytic set}\label{sing}

Let $X\subset M$ be a $C$-analytic set and let $\widetilde{X}$ be a complexification of $X$. We define the \em singular locus of $X$ \em as $\Sing(X):=\Sing(\widetilde{X})\cap M$. Its complement $\Reg(X):=X\setminus\Sing(X)$ is the set of \em regular points of $X$\em. Observe that $\Sing(X)$ is a $C$-analytic set of strictly smaller dimension than $X$. We define inductively $\Sing_\ell(X)=\Sing(\Sing_{\ell-1}(X))$ for $\ell\geq1$ where $\Sing_1(X)=\Sing(X)$. In particular, $\Sing_\ell(X)=\varnothing$ if $\ell\geq\dim(X)+1$. We will write $X:=\Sing_0(X)$ for simplicity.

\subsubsection{Irreducible components of a $C$-analytic set}
A $C$-analytic set is \em irreducible \em if it is not the union of two $C$-analytic sets different from itself. In addition $X$ is irreducible if and only if it admits a fundamental system of invariant irreducible complexifications. Given a $C$-analytic set $X$, there is a unique irredundant (countable) locally finite family of irreducible $C$-analytic sets $\{X_i\}_{i\geq1}$ such that $X=\bigcup_{i\geq1}X_i$. The $C$-analytic sets $X_i$ are called the \em irreducible components of $X$\em. For further details see \cite{wb}.

\subsection{Normalization of complex analytic spaces}\label{norm}
One defines the normalization of a complex analytic space in the following way \cite[VI.2]{n}. A complex analytic space $(X,\an_X)$ is \em normal \em if for all $x\in X$ the local analytic ring $\an_{X,x}$ is reduced and integrally closed. A \em normalization \em $(Y,\pi)$ of a complex analytic space $(X,\an_X)$ is a normal complex analytic space $(Y,\an_Y)$ together with a proper surjective holomorphic map $\pi:Y\to X$ with finite fibers such that $Y\setminus\pi^{-1}(\Sing(X))$ is dense in $Y$ and $\pi|:Y\setminus\pi^{-1}(\Sing(X))\to X\setminus\Sing(X)$ is an analytic isomorphism. The normalization $(Y,\pi)$ of a reduced complex analytic space $X$ always exists and is unique up to isomorphism \cite[VI.2.Lem.2 \& VI.3.Thm.4]{n}. 

\begin{remark}\label{norm1}
If $\Omega$ is an open subset of $X$, then $(\pi^{-1}(\Omega),\pi|)$ is the normalization of $\Omega$. If $Z$ is a connected component of $\pi^{-1}(\Omega)$, the map $\pi|_Z:Z\to \Omega$ is proper and $\pi(Z)$ is by Remmert's Theorem \cite[VII.\S2.Thm.2]{n} a complex analytic subspace of $\Omega$. In addition $Z$ is irreducible because it is connected and normal. Thus, $Z\setminus(\pi|_Z)^{-1}(\Sing(\pi(Z)))$ is by \cite[IV.1.Cor.2]{n} connected. Consequently,
$$
\pi(Z\setminus(\pi|_Z)^{-1}(\Sing(\pi(Z))))=\pi(Z)\setminus\Sing(\pi(Z))=\Reg(\pi(Z))
$$
is connected and $\pi(Z)$ is irreducible by \cite[IV.1.Cor.1]{n}. Thus, $\pi(Z)$ is an irreducible component of $Z$ and $(Z,\pi|_Z)$ is the normalization of $\pi(Z)$.
\end{remark}

\subsection{Analytic functions on a semianalytic set}

Let $T\subset M$ be a subset and let $U\subset M$ be an open neighborhood of $T$. Denote the collection of all analytic functions on $U$ that vanish identically on $T$ with $\ideal(T,U):=\{f\in\an(U):\ f|_T\equiv0\}$. The \em Zariski closure $\ol{T}^{\zar}_U$ of $T$ in $U$ \em is the intersection
$$
\ol{T}^{\zar}_U:=\bigcap_{f\in\ideal(T,U)}\ceros(f)
$$ 
of the zerosets $\ceros(f):=\{x\in U:\ f(x)=0\}$ where $f\in\ideal(T,U)$. 

\subsubsection{Sheaf of analytic germs on a semianalytic set}
Fix a semianalytic set $S\subset M$ and an open neighborhood $U\subset M$ of $S$. Consider the coherent sheaf of ideals $\J_U:=\ideal(S,U)\an_U$. As the sequence
\begin{equation}\label{exactseq}
0\to\J_U\to\an_U\to\an_U/\J_U\to 0
\end{equation}
is exact, the sheaf of rings $\an_U/\J_U$ is coherent. Observe that $H^0(U,\J_U)=\ideal(S,U)$. The exact sequence \eqref{exactseq} induces by Cartan's Theorem B an exact sequence on global sections
$$
0\to\ideal(S,U)\to\an(U)\to H^0(U,\an_U/\J_U)\to 0.
$$
Consequently $H^0(U,\an_U/\J_U)\cong\an(U)/\ideal(S,U)$.

The sheaf $\an_M|_S$ is by \cite[Cor. (I,8)]{f} coherent. Observe that $H^0(S,\an_M|_S)$ is the ring of germs at $S$ of analytic functions on $M$ and that $\an_M|_S=\an_U|_S$ for each open neighborhood $U\subset M$ of $S$. Consider the ideal
$$
\ideal(S):=\{f\in H^0(S,\an_M|_S):\ f|_S=0\}.
$$
Define the sheaf of ideals $\J_S:=\ideal(S)\an_M|_S$, which is by \cite[Cor.(I,8)]{f} a coherent sheaf of ideals of $\an_M|_S$. Consider the exact sequence
\begin{equation}\label{exactseq2}
0\to \J_S\to\an_M|_S\to\an_S:=\an_M|_S/\J_S\to 0.
\end{equation}
As $\J_S$ and $\an_M|_S$ are coherent, the quotient sheaf $\an_S$ is also coherent. In addition $H^0(S,\J_S)=\ideal(S)$. The exact sequence \eqref{exactseq2} induces by Cartan's Theorem B an exact sequence on global sections:
$$
0\to\ideal(S)\to H^0(S,\an_M|_S)\to H^0(S,\an_S)\to 0.
$$

\subsubsection{Ring of analytic functions on a semianalytic set}
The ring 
\begin{equation}\label{os}
\an(S):=H^0(S,\an_S)\cong H^0(S,\an_M|_S)/\ideal(S)
\end{equation}
is constituted by those functions $f:S\to\R$ that admit analytic extensions ${f'}:W\to\R$ to an open neighborhood $W\subset M$ of $S$. If $U\subset M$ is an open neighborhood of $S$, the restriction homomorphism
$$
\psi_U:\an(\ol{S}^{\zar}_U)\cong\an(U)/\ideal(S,U)\to\an(S),\ f\mapsto f|_S
$$
is injective. If $U\subset V$ are open neighborhoods of $S$ in $M$, the restriction homomorphism
$$
\rho_{UV}:\an(\ol{S}^{\zar}_V)\cong\an(V)/\ideal(S,V)\to\an(\ol{S}^{\zar}_U)\cong\an(U)/\ideal(S,U),\ f\mapsto f|_U
$$
is injective. It holds: \em
\begin{equation}\label{osdl}
\an(S)\cong\displaystyle\lim_{\substack{\longrightarrow\\S\subset U\subset M}}\an(U)/\ideal(S,U)=\displaystyle\lim_{\substack{\longrightarrow\\S\subset U\subset M}}\an(\ol{S}^{\zar}_U)
\end{equation}
where $U$ runs over the open neighborhoods of $S$ in $M$\em. 

\subsection{Some algebraic properties of saturated ideals}
An ideal $\gta$ of $\an(S)$ is \em saturated \em if it coincides with its \em saturation\em
$$
\widetilde{\gta}:=\{f\in\an(S):\ f_x\in\gta\an_{S,x}\ \forall x\in S\}=H^0(S,\gta\an_S)
$$
Some basic immediate properties are the following:
\begin{itemize}
\item[(i)] $\gta\subset\widetilde{\gta}$.
\item[(ii)] If $\gta\subset\gtb$, then $\widetilde{\gta}\subset\widetilde{\gtb}$.
\item[(iii)] $\widetilde{\widetilde{\gta}}=\widetilde{\gta}$.
\end{itemize}

We state next some results concerning primary ideals and primary decomposition of saturated ideals of the ring $\an(S)$. The proofs are by equality \eqref{osdl} similar to the corresponding ones for the ring $\an(X)$ where $X$ is a $C$-analytic set. Thus, we only refer next to the corresponding result for $\an(X)$ and we do not provide further details. Recall that a family of ideals $\{\gta_i\}_{i\in I}$ of $\an(S)$ is \em locally finite \em if the family of their zerosets $\{\ceros(\gta_i)\}_{i\in I}$ is locally finite in $S$.

\begin{lem}\label{basic}
Let $\gtq\subset\an(S)$ be a primary ideal.
\begin{itemize}
\item[(i)] Suppose $\ceros(\gtq)\neq\varnothing$. Then $f\in\an(S)$ belongs to $\gtq$ if and only if
the germ $f_x$ belongs to the stalk $\gtq\an_{S,x}$ for some $x\in\ceros(\gtq)$.
\item[(ii)] The primary ideal $\gtq$ is saturated if and only if $\ceros(\gtq)\neq\varnothing$.
\end{itemize}
\end{lem}
\begin{proof}
The proof is analogous to that of \cite[2.1]{bp}.
\end{proof}

\begin{lem}\label{intpri}
Let $\{\gta_i\}_{i\in I}$ be a non empty locally finite collection of saturated ideals of $\an(S)$. Let $\gtp$ be a saturated prime ideal of $\an(S)$ such that $\bigcap_{i\in I}\gta_i\subset\gtp$. Then there exists an index $i\in I$ such that $\gta_i\subset\gtp$. 
\end{lem}
\begin{proof}
The proof is similar to the one of \cite[2.2.10]{db}.
\end{proof}

\begin{defns}
A locally finite decomposition $\gta=\bigcap_{i\in I}\gtq_i$ where $\{\gtq_i\}_{i\in I}$ is a locally finite family of saturated primary ideals of $\an(S)$ is said a \em primary decomposition \em if it is irredundant and the associated primes $\gtp_i:=\sqrt{\gtq_i}$ are pairwise different.

A primary ideal $\gtq_k\in\{\gtq_i\}_{i\in I}$ is called an \em isolated primary component \em if $\gtp_k:=\sqrt{\gtq_k}$ is minimal among the prime ideals $\{\gtp_i:=\sqrt{\gtq_i}\}_{i\in I}$. Otherwise, it is called an
\em immersed primary component.
\end{defns}

\begin{lem}\label{pridecomp}
Let $\gta$ be a saturated ideal of $\an(S)$. Then $\gta$ admits a locally finite primary decomposition.
\end{lem}
\begin{proof}
The proof is conducted similar to the one of \cite[2.3.6]{db}.
\end{proof}

\begin{lem}
Let $\gta=\bigcap_{i\in I}\gtq_i$ be a primary decomposition for a saturated ideal $\gta$ of $\an(S)$. Then the prime ideals $\gtp_i:=\sqrt{\gtq_i}$ and the isolated primary components $\gtq_i$ are uniquely determined by $\gta$, that is, they do not depend on the primary decomposition.
\end{lem}
\begin{proof}
The proof is similar to that of \cite[2.8]{bp}.
\end{proof}

Primary decompositions enjoy the good behavior one can expect for radical ideals. One proves straightforwardly the following. 

\begin{cor}\label{lemmaidreal}
Let $\gta\subset\an(S)$ be a saturated radical ideal and let $\gta=\bigcap_{i\in I}\gtq_i$ be a primary decomposition of $\gta$. Then each $\gtq_i$ is prime and the primary decomposition is unique.
\end{cor}

\section{Tame $C$-semianalytic sets}\label{s3}

The class of amenable $C$-semianalytic sets is the smallest family of subsets of $M$ that contains $C$-analytic sets, open $C$-semianalytic sets and is closed for finite unions and intersections. A first example of amenable $C$-semianalytic sets is basic $C$-semianalytic sets, that is, sets of the type $S:=\{f=0,g_1>0,\ldots,g_s>0\}$ where $f,g_1,\ldots,g_s\in\an(M)$. Consequently, global $C$-semianalytic sets and semialgebraic subsets of $\R^n$ are amenable $C$-semianalytic sets. We present next some enlightening examples.

\begin{example}[$C$-semianalytic set that is not amenable in any of its neighborhoods]\label{tamenot}
For each $k\geq0$ consider the basic $C$-semianalytic set 
$$
S_k:=\{z^2-(y-k)x^2=0, y\geq k\}\subset\R^3. 
$$
Observe that the family $\{S_k\}_{k\geq0}$ is locally finite, so $S:=\bigcup_{k\geq0}S_k$ is a $C$-semianalytic set. 

\begin{center}
\begin{figure}[ht]
\hspace*{-3mm}\includegraphics[width=.50\textwidth]{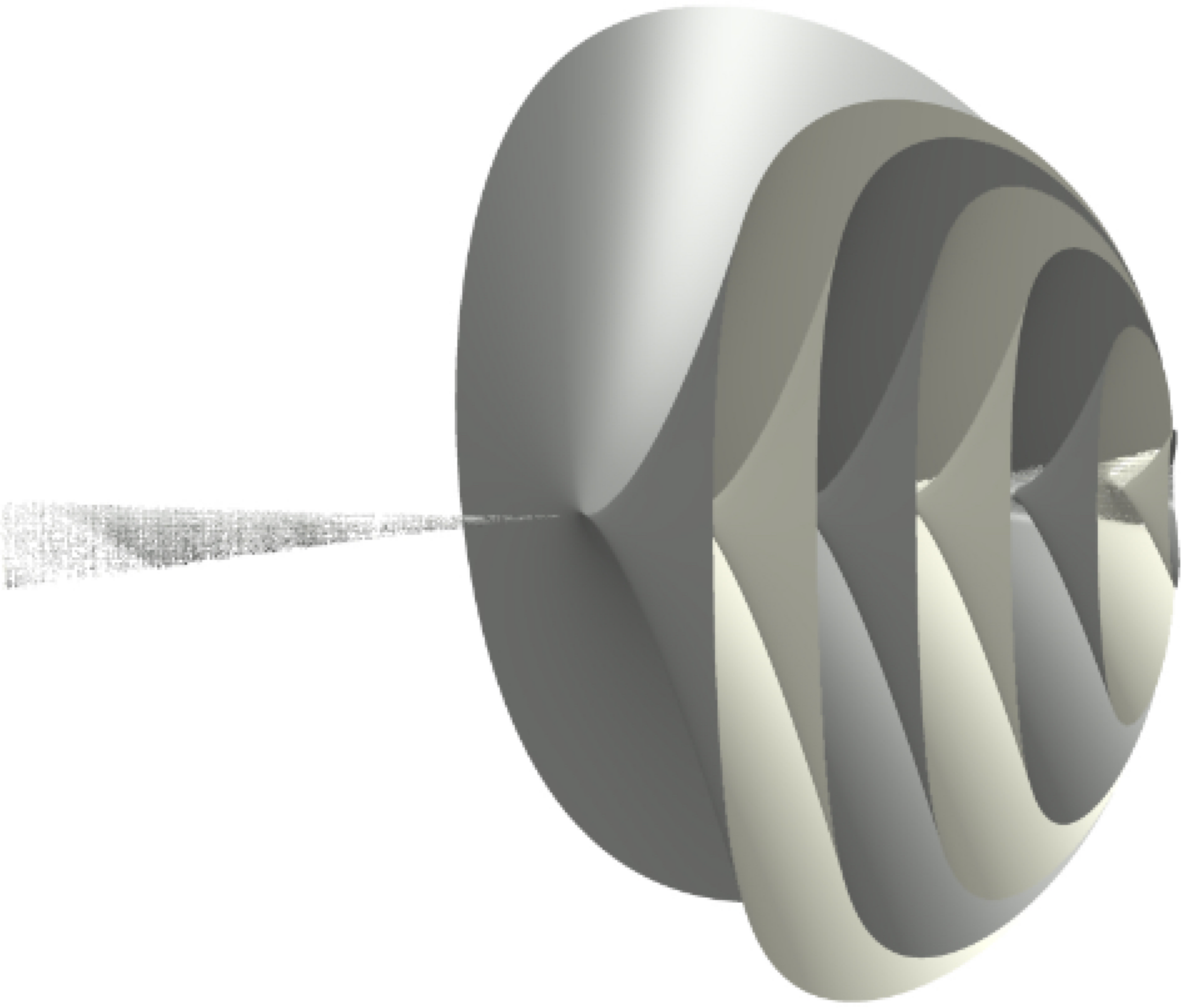}
\caption{$S:=\bigcup_{n\geq0}S_n$ from Example \ref{tamenot}}
\end{figure}
\end{center}

\vspace{-7.5mm}
Observe that $S$ is not an amenable $C$-semianalytic subset of any open neighborhood $U$ of $S$ in $\R^3$ because at the points $\{(0,0,y):\ y\geq0\}$ the family of the Zariski closures $\{\ol{S}^{\zar}_{k,U}\}_{n\geq1}$ is not locally finite, so $\ol{S}^{\zar}=\R^3$.
\end{example}

\begin{example}[$C$-semianalytic set that is locally $C$-analytic but it is not amenable]\label{tamenot0}
Let 
$$
S_{kj}:=\Big\{x^2+2\Big(1+\frac{1}{j}\sin(y)\Big)xz+z^2=0, 2k\pi\leq y\leq(2k+1)\pi\Big\}\cup\{x=0,z=0\}\subset\R^3
$$
for $j\geq 1$ and $k\in\Z$ (see \cite[\S11]{wb}). 
\begin{center}
\begin{figure}[ht]
\hspace*{-3mm}\includegraphics[width=.85\textwidth]{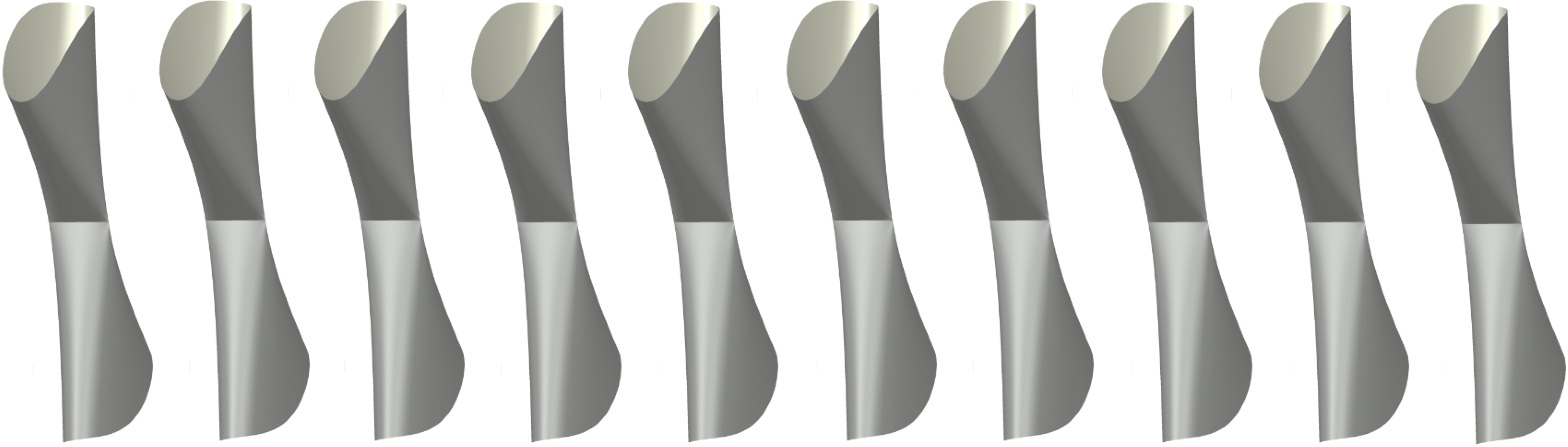}
\caption{Zariski closure $X_j$ of $S_{\ell j}$ (Example \ref{tamenot0})}
\end{figure}
\end{center}
\vspace*{-7mm}
The Zariski closure of $S_{\ell j}$ is $X_j:=\bigcup_{k\in\Z}S_{kj}$. Define $S:=\bigcup_{j\geq1}S_{jj}$. It holds that for each $x\in M$ there exists an open neigborhood $U^x$ such that $S\cap U^x=\{f=0\}\cap U^x$ for some $f\in\an(\R^3)$. As $\bigcup_{j\geq1}X_j\subset\ol{S}^{\zar}$, we conclude that $\ol{S}^{\zar}=\R^3$. Consequently, $S$ is not an amenable $C$-semianalytic set.
\end{example}

\begin{example}[Tame $C$-semianalytic set that is not a global $C$-semianalytic set]\label{buttame}
\ Consider the $C$-semi\-analytic set $S:=\bigcup_{k\geq1} S_k$ where
$$
S_k:=\{0<x<k, 0<y<1/k\}\subset\R^2. 
$$
As $S$ is an open $C$-semianalytic set, it is amenable. We claim: \em $S$ is not a global $C$-semianalytic set\em.

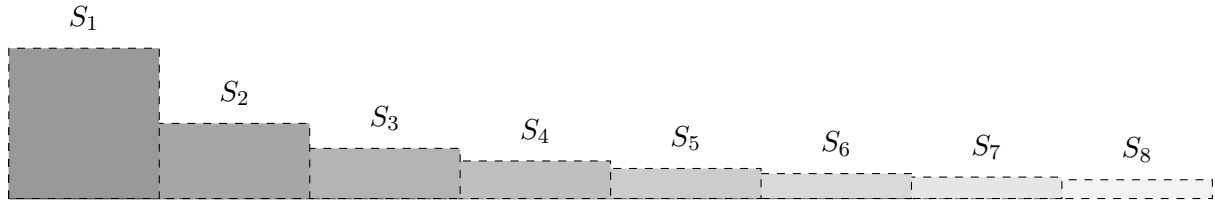
\begin{figure}[ht]
\centering
\begin{tikzpicture}[x=.4cm,y=.4cm]

\draw[dashed,fill=black!5!white] (0,0) -- (0,0.625) -- (40,0.625) -- (40,0) -- (0,0);
\draw[dashed,fill=black!10!white] (0,0) -- (0,0.714) -- (35,0.714) -- (35,0) -- (0,0);
\draw[dashed,fill=black!15!white] (0,0) -- (0,0.83) -- (30,0.83) -- (30,0) -- (0,0);
\draw[dashed,fill=black!20!white] (0,0) -- (0,1) -- (25,1) -- (25,0) -- (0,0);
\draw[dashed,fill=black!25!white] (0,0) -- (0,1.25) -- (20,1.25) -- (20,0) -- (0,0);
\draw[dashed,fill=black!30!white] (0,0) -- (0,1.67) -- (15,1.67) -- (15,0) -- (0,0);
\draw[dashed,fill=black!35!white] (0,0) -- (0,2.5) -- (10,2.5) -- (10,0) -- (0,0);
\draw[dashed,fill=black!40!white] (0,0) -- (0,5) -- (5,5) -- (5,0) -- (0,0);

\draw (2.5,6) node{$S_1$};
\draw (7.5,3.5) node{$S_2$};
\draw (12.5,2.67) node{$S_3$};
\draw (17.5,2.25) node{$S_4$};
\draw (22.5,2) node{$S_5$};
\draw (27.5,1.83) node{$S_6$};
\draw (32.5,1.714) node{$S_7$};
\draw (37.5,1.625) node{$S_8$};

\end{tikzpicture}
\caption{Tame $C$-semianalytic set $S:=\bigcup_{k\geq1} S_k$ (Example \ref{buttame})}
\end{figure}

Otherwise there exist finitely many analytic functions $g_{ij}\in\an(M)$ not all identically zero such that $S=\bigcup_{i=1}^r\{g_{i1}>0,\ldots,g_{is}>0\}$. Notice that the boundary of $S$ is contained in the the $C$-analytic set $X=\bigcup_{i=1}^r\bigcup_{j=1}^s\{g_{ij}=0\}$. Thus, there exist indices $i,j$ such that $g_{ij}$ is not identically zero but it vanishes on infinitely many segments of the type $(a_k-\veps_k,a_k+\veps_k)\times\{1/k\}$ for different $k$ where $a_k\in(k,k+1)$ and $\veps_k>0$. But this implies that $g_{ij}$ vanishes identically in infinitely many lines of the type $y=1/k$, so $g_{ij}=0$, which is a contradiction. Thus, $S$ is not a global $C$-semianalytic set.
\end{example}

As one can expect $C$-semianalytic sets that are open subsets of $C$-analytic sets are amenable.

\begin{lem}\label{openX}
Let $X\subset M$ be a $C$-analytic set and let $S\subset X$ be a $C$-semianalytic set that is open in $X$. Then there exists an open $C$-semianalytic set $W$ such that $S=X\cap W$. In particular, $S$ is an amenable $C$-semianalytic set.
\end{lem}
\begin{proof}
Fix a point $x\in X$ and let $U^x$ be an open neighborhood of $x$ in $M$ such that $S\cap U^x$ is a global $C$-semianalytic set. By \cite[3.1]{abs} there exists a finite union $W^x$ (maybe empty) of open basic $C$-semianalytic subsets of $M$ such that $S\cap U^x=X\cap W^x$. Denote $W:=(M\setminus X)\cup\bigcup_{x\in X}W^x$, which is an open $C$-semianalytic set. Observe that $S=X\cap W$, as required.
\end{proof}

\subsection{Basic properties of amenable $C$-semianalytic sets}\label{ucc}
The family of amenable $C$-semianalytic sets is closed under: 
\begin{itemize}
\item finite unions, 
\item finite intersections, 
\item inverse image under analytic maps between real analytic manifolds, 
\item taking interior, 
\item considering connected components or unions of some of them,
\item set of points of pure dimension $k$.
\end{itemize}
\begin{proof}
The first two properties are clear. For the remaining ones we proceed as follows. \setcounter{paragraph}{0}

\paragraph{}\em Inverse image\em. Let $f:M\to N$ be an analytic map between real analytic manifolds and let $S\subset N$ be an amenable $C$-semianalytic set. Then $f^{-1}(S)\subset M$ is an amenable $C$-semianalytic set.

We may assume that $S=X\cap W$ where $X$ is a $C$-analytic subset of $N$ and $W$ is an open $C$-semianalytic subset of $N$. As $f^{-1}(S)=f^{-1}(X)\cap f^{-1}(W)$ and $f^{-1}(X)$ is a $C$-analytic subset of $M$ and $f^{-1}(W)$ is an open $C$-semianalytic subset of $M$, as required.

\paragraph{}\em Interior\em. Let $S\subset M$ be an amenable $C$-semianalytic set and let $X$ be its Zariski closure. It holds that $\Int_X(S)=S\setminus\cl_X(X\setminus S)=S\cap U$ where $U:=M\setminus\cl_X(X\setminus S)$ is an open $C$-semianalytic set. Then $\Int_X(S)$ is amenable.

\paragraph{}\em Connected components\em. Let $S\subset M$ be an amenable $C$-semianalytic set and let $\{S_i\}_{i\geq1}$ be the family of the connected components of $S$. We know by \cite[2.7]{bm} that $\{S_i\}_{i\geq1}$ is a locally finite family, so the connected components of $S$ are open and closed subsets of $S$. Let ${\mathfrak F}\subset\{i\geq1\}$ be any non-empty subset. We claim that $T:=\bigcup_{i\in{\mathfrak F}}S_i$ is an amenable $C$-semianalytic set. Notice that: 
$$
T=S\cap\Big(M\setminus\bigcup_{i\not\in{\mathfrak F}}\cl(S_i)\Big)=S\cap U
$$
where $U:=M\setminus\bigcup_{i\not\in{\mathfrak F}}\cl(S_i)$ is an open $C$-semianalytic set. This last claim follows because the family $\{S_i\}_{i\not\in{\mathfrak F}}$ is locally finite and each $\cl(S_i)$ is a $C$-semianalytic set \cite[3.A]{abf1}. Consequently $T$ is an amenable $C$-semianalytic set.

\paragraph{}\em Set of points of pure dimension $k$\em. Let $S^*_{(k)}$ be the subset of points $x\in S$ such that the germ $S_x$ is pure dimensional and $\dim(S_x)=k$. We have $S^*_{(k)}=S\cap(M\setminus\bigcup_{j\neq k}\cl(S_{(j)}))$ for each $k\geq0$ where $S_{(j)}$ is the set of points of $S$ of local dimension $j$ for $j\geq0$. Consequently, $S^*_{(k)}$ is amenable $C$-semianalytic.
\end{proof}

\begin{remarks}
(i) A locally finite union of amenable $C$-semianalytic sets is not in general an amenable $C$-semianalytic set. Each $C$-semianalytic set is a locally finite union of basic $C$-semianalytic sets, but there are many $C$-semianalytic sets that are not amenable. 

(ii) Let $U:=\R^3\setminus S$ where $S$ is the closed $C$-semianalytic set defined in Example \ref{tamenot}. As $U$ is an open $C$-semianalytic set, it is amenable. However is complement $S=\R^3\setminus U$ is not amenable.

(iii) As we will see in the Example \ref{dimktamenot}, the ($C$-semianalytic) set of points of local dimension $k$ of an amenable $C$-semianalytic set may not to be amenable.
\end{remarks}

\begin{example}[Bad behavior of the closure of amenable $C$-semianalytic sets]\label{ex}
(i) Consider the open $C$-semianalytic set:
$$
S:=\bigcup_{k\geq1}S_k\quad\text{where $S_k:=\{y<kx+z,y>kx-z,k<x<k+1,0<z<1\}\subset\R^3$}
$$
Clearly $S$ is an amenable $C$-semianalytic set. We have
$$
\cl(S):=\bigcup_{k\geq1}\cl(S_k)\quad\text{where $\cl(S_k):=\{y\leq kx+z,y\geq kx-z,k\leq x\leq k+1,0\leq z\leq1\}$}
$$
Consider the $C$-analytic set $X:=\{z=0\}$ and observe that 
$$
\cl(S)\cap X:=\bigcup_{k\geq1}\cl(S_k)\cap X\quad\text{where $\cl(S_k)\cap X:=\{y=kx,k\leq x\leq k+1,z=0\}$}
$$
which is not an amenable $C$-semianalytic set. To that end observe that $\ol{\cl(S)\cap X}^{\zar}=\{z=0\}$ has dimension $2$ while $\cl(S)\cap X$ has dimension $1$. As $X$ is amenable, $\cl(S)$ is not amenable.

(ii) The closure $T:=\cl(S)$ in $\R^2$ of the amenable $C$-semianalytic set $S$ in Example \ref{buttame} is not an amenable $C$-semianalytic set. To prove this one can apply Theorem \ref{algor1} below.
\end{example}

\subsection{Characterization of amenable $C$-semianalytic sets}
In \cite{abf1} we show that $C$-semi\-analytic set are countable locally finite unions of basic $C$-semianalytic sets. Here we obtain the corresponding result for amenable $C$-semianalytic sets. 

\begin{prop}\label{normal2}
Let $S\subset M$ be a $C$-semianalytic set. Then $S$ is amenable if and only if there exists a countable locally finite family of basic $C$-semianalytic sets $\{S_i\}_{i\geq1}$ such that $S=\bigcup_{i\geq1}S_i$ and the family $\{\ol{S_i}^{\zar}\}_{i\geq1}$ is locally finite (after eliminating repetitions).
\end{prop}

Before proving Proposition \ref{normal2} we need some preliminary work. The following result is a kind of weak version of Proposition \ref{normal2} for open $C$-semianalytic sets. 

\begin{lem}\label{odes}
Let $U\subset M$ be an open $C$-semianalytic set. Then there exists a locally finite countable family $\{U_k\}_{k\geq1}$ of open basic $C$-semianalytic sets such that $U=\bigcup_{j\geq1}U_j$.
\end{lem}
\begin{proof}
It is enough to prove: \em there exists a locally finite countable family $\{U_k\}_{k\geq1}$ of open global $C$-semianalytic sets such that $U=\bigcup_{j\geq1}U_j$\em.
 
For each $x\in M$ pick an open neighborhood $U^x$ such that $U\cap U^x$ is an open global $C$-semianalytic set. As $M$ is paracompact and it admits countable exhaustion by compact sets, there exists a countable locally finite open refinements ${\mathscr W}:=\{W_j\}_{j\geq1}$ and ${\mathscr W}':=\{W_j'\}_{j\geq1}$ of ${\mathscr U}:=\{U^x\}_{x\in M}$ such that $\cl(W_j')\subset W_j$ for each $j\geq1$. As the the closed sets $\cl(W_j')$ and $M\setminus W_j$ are disjoint, there exists a continuous function $f_j:M\to\R$ such that $f_j|_{\cl(W_j')}\equiv 1$ and $f_j|_{M\setminus W_j}\equiv-1$. Let $g_j$ be an analytic approximation of $f_j$ such that $|g_j-f_j|<\frac{1}{2}$. Notice that
$$
\cl(W_j')\subset V_j:=\{g_j>0\}\subset W_j
$$ 
for each $j\geq1$. Thus, the family $\{U_j:=U\cap V_j\}_{j\geq1}$ is countable, locally finite, its members are open global $C$-semianalytic sets and satisfies $U=\bigcup_{j\geq1}U_j$, as required.
\end{proof}

The following result allows to reduce the proof of Proposition \ref{normal2} to the advantageous case of open $C$-semianalytic subsets of real analytic manifolds. Its proof is quite similar to and less technical than the one of Lemma \ref{pd} below. We leave the concrete details to the reader.

\begin{lem}\label{decomp}
Let $S\subset M$ be an amenable $C$-semianalytic set. Then there exists finitely many $C$-analytic sets $Z_1,\ldots,Z_r$ and finitely many open $C$-semianalytic sets $U_1,\ldots,U_r$ such that $S=\bigsqcup_{i=1}^r(Z_i\cap U_i)$, each $Z_i\cap U_i$ is a real analytic manifold, $Z_i$ is the Zariski closure of $Z_i\cap U_i$ and $\dim(Z_{i+1}\cap U_{i+1})<\dim(Z_i\cap U_i)$ for $i=1,\ldots,r-1$.
\end{lem}

We are ready to prove Proposition \ref{normal2}.

\begin{proof}[Proof of Proposition \em\ref{normal2}]
We prove first the `only if' implication. By Lemma \ref{decomp} we may assume $S=Z\cap U$ where $Z$ is a $C$-analytic set, $U$ is an open $C$-semianalytic set, $S$ is a real analytic manifold and $Z$ is the Zariski closure of $S$. Let $\{C_{\ell}\}_{\ell\geq1}$ be the collection of the connected components of $Z\cap U$. Notice that the Zariski closure $\ol{C_{\ell}}^{\zar}$ is an irreducible component of $Z$. Each $C$-semianalytic set $C_{\ell}$ is an open subset of $Z$, so by Lemma \ref{openX} there exists an open $C$-semianalytic set $V_{\ell}$ such that $C_\ell=Z\cap V_\ell$. By Lemma \ref{odes} we write $V_\ell$ as a countable locally finite union $V_\ell=\bigcup_{j\geq1}V_{\ell j}$ of open basic $C$-semianalytic sets. Consequently $S_{\ell j}:=Z\cap V_{\ell j}$ is either empty or a basic $C$-semianalytic set whose Zariski closure is an irreducible component of $Z$. The collection $\{S_{\ell,j}\}_{\ell,j\geq1}$ satisfies the required conditions. 

We prove next the `if' implication by induction on the dimension of $S$. If $\dim(S)=0$, then $S$ is a $C$-analytic set and in particular it is amenable. Assume that the result is true if $\dim(S)<d$ and let us check that it is also true for dimension $d$. 

Consider the $C$-analytic set $X:=\bigcup_{i\geq1}\ol{S_i}^{\zar}$ and let $X'$ be the union of $\Sing(\ol{S_i}^{\zar})$ and the irreducible components of $\ol{S_i}^{\zar}$ of dimension $<d$ for $i\geq1$. It holds that $X'$ is a $C$-analytic set of dimension $<d$. Consequently, each intersection $S_i\cap X'$ is a basic $C$-semianalytic set of dimension $<d$. In addition, the countable family $\{S_i\cap X'\}_{i\geq1}$ is locally finite and $\{\ol{S_i\cap X'}^{\zar}\}_{i\geq1}$ is locally finite (after eliminating repetitions). By induction hypothesis $S\cap X'=\bigcup_{i\geq1}S_i\cap X'$ is an amenable $C$-semianalytic set. It only remains to check: \em $S\setminus X'$ is an amenable $C$-semianalytic set\em. 

Notice that $S_i\setminus X'$ is either empty or a basic $C$-semianalytic set of dimension $d$. We claim: \em $S_i\setminus X'$ is an open subset of the real analytic manifold $X\setminus X'$ for each $i\geq1$\em, so the $C$-semianalytic set $S\setminus X'=\bigcup_{i\geq1}S_i\setminus X'$ is an open subset of $X$. Consequently, $S\setminus X'$ is by Lemma \ref{openX} an amenable $C$-semianalytic set.

Indeed, write $S_i=\ol{S_i}^{\zar}\cap V_i$ where $V_i$ is a basic open $C$-semianalytic set. As $\Sing(\ol{S_i}^{\zar})$ and the irreducible components of $\ol{S_i}^{\zar}$ of dimension strictly smaller than $d$ are contained in $X'$, the basic $C$-semianalytic set $S_i\setminus X'=(\ol{S_i}^{\zar}\setminus X')\cap V_i$ is an open subset of the real open analytic manifold $(\ol{S_i}^{\zar}\setminus X')$. As $\ol{S_i}^{\zar}\setminus X'$ is in turn an open subset of the real analytic manifold $X\setminus X'$, the claim holds, as required. 
\end{proof}

As a consequence of Proposition \ref{normal2} we find a sufficient condition under which a countable locally finite union of amenable $C$-semianalytic sets keeps amenable.

\begin{cor}\label{lfuok2}
Let $\{S_i\}_{i\geq1}$ be a locally finite collection of amenable $C$-semianalytic sets such that the family $\{\ol{S_i}^{\zar}\}_{i\geq1}$ is locally finite (after eliminating repetitions). Then $S=\bigcup_{i\geq1}S_i$ is an amenable $C$-semianalytic set.
\end{cor} 
\begin{proof}
For each $i\geq1$ there exists by Proposition \ref{normal2} a countable family $\{S_{ij}\}_{j\geq1}$ of basic $C$-semianalytic sets such that $S_i=\bigcup_{j\geq1}S_{ij}$ and the family $\{\ol{S_{ij}}^{\zar}\}_{j\geq1}$ is locally finite (after eliminating repetitions). Notice that $\{S_{ij}\}_{i,j\geq1}$ is a countable family of basic $C$-semianalytic sets. As $\ol{S_{ij}}^{\zar}\subset\ol{S_i}^{\zar}$ for each $i,j\geq1$ and the families $\{\ol{S_i}^{\zar}\}_{i\geq1}$ and $\{\ol{S_{ij}}^{\zar}\}_{j\geq1}$ are locally finite (after eliminating repetitions), we conclude that the family $\{\ol{S_{ij}}^{\zar}\}_{i,j\geq1}$ is also locally finite (after eliminating repetitions). By Proposition \ref{normal2} we conclude that $S=\bigcup_{i\geq1}S_i=\bigcup_{i,j\geq1}S_{ij}$ is amenable, as required.
\end{proof}

\subsection{Images of amenable $C$-semianalytic sets under proper holomorphic maps}\setcounter{paragraph}{0} 
Before proving Theorem \ref{properint-tame} we need some preliminary work. Let $(X,\an_X)$ be a reduced Stein space endowed with an anti-involution $\sigma$ such that its fixed part space $X^\sigma$ is non-empty. Let ${\mathcal A}(X)\subset\an(X)$ be the subring of all invariant holomorphic sections of $\an(X)$ and 
$$
{\mathcal A}(X^\sigma):=\{F|_{X^\sigma}:\ F\in{\mathcal A}(X)\}\subset\an(X^\sigma). 
$$
A $C$-semianalytic set $S\subset X^\sigma$ is \em ${\mathcal A}(X^\sigma)$-definable \em if if for each $x\in X^\sigma$ there exists an open neighborhood $U^x$ such that $S\cap U^x$ is a finite union of sets of the type $\{f=0,g_1>0,\ldots,g_r>0\}$ where $f,g_i\in{\mathcal A}(X^\sigma)$.

\begin{lem}\label{zarx}
Let $Y$ be a $C$-analytic subset of $X^\sigma$ and let $Z$ be the Zariski closure of $Y$ in $X$. Then $\dim_\C(Z)=\dim_\R(Y)$, $\Sing(Y)\subset\Sing(Z)$ and $Y\setminus\Sing(Z)$ is a union of connected components of $X^\sigma\setminus Z$.
\end{lem}
\begin{proof}
If $\dim_\R(Y)<\dim_\C(Z)$, it holds $Y\subset\Sing(Z)\subsetneq Z$, which is a contradiction. Thus, $d:=\dim_\C(Z)=\dim_\R(Y)$. The inclusion $\Sing(Y)\subset\Sing(Z)$ is clear, so $Y\setminus\Sing(Z)$ is a real analytic manifold of dimension $d$. In addition, $X^\sigma\setminus Z$ is a real analytic manifold of dimension $d$ and $Y$ is a $C$-analytic subset of $X^\sigma\cap Z$. Thus, $Y\setminus\Sing(Z)$ is an open and closed subset of $X^\sigma\setminus Z$, so $Y\setminus\Sing(Z)$ is a union of connected components of $X^\sigma\setminus Z$, as required. 
\end{proof}

We develop next a useful presentation of ${\mathcal A}(X^\sigma)$-definable and amenable $C$-semianalytic sets.
\begin{lem}\label{pd}
Let $S\subset X^\sigma$ be an ${\mathcal A}(X^\sigma)$-definable and amenable $C$-semianalytic set. Then there exist finitely many invariant complex analytic subsets $Z_1,\ldots,Z_r$ of $X$ and finitely many open ${\mathcal A}(X^\sigma)$-definable $C$-semianalytic subsets $U_1,\ldots,U_r$ of $X^\sigma$ such that $S=\bigsqcup_{i=1}^r(Z_i\cap U_i)$, each $Z_i\cap U_i$ is a real analytic manifold and $\dim(Z_{i+1}\cap U_{i+1})<\dim(Z_i\cap U_i)$ for $i=1,\ldots,r-1$.
\end{lem}
\begin{proof}
We proceed by induction on the dimension $d$ of $S$. If $S$ has dimension $0$, the result is clearly true. Assume the result true if the dimension of $S$ is $<d$ and let us check that it is also true if $S$ has dimension $d$.
 
\paragraph{}As $S$ is an amenable $C$-semianalytic set, there exist $C$-analytic sets $Y_1,\ldots,Y_s$ and open $C$-semianalytic sets $V_1,\ldots,V_s$ such that $S=\bigcup_{i=1}^s(Y_i\cap V_i)$. Denote $Y:=\bigcup_{i=1}^sY_i$, which has dimension $d$. Let $Y'$ be the union of the irreducible components of $Y$ of dimension $d$ and let $Y''$ the union of those of smaller dimension. Let $Z'$ be the Zariski closure of $Y'$ in $X$ and $Z''$ the Zariski closure of $Y''$ in $X$. Then $Z=Z'\cup Z''$ is the Zariski closure of $Y$ in $X$.

By Lemma \ref{zarx} the $C$-semianalytic set $Y'\setminus(\Sing(Z')\cup Z'')$ is a union of connected components of $(X^\sigma\cap Z)\setminus((\Sing(Z')\cup Z''))$. By \cite[Prop.3.5]{abf1} $Y'\setminus(\Sing(Z')\cup Z'')$ is a ${\mathcal A}(X^\sigma)$-definable $C$-semianalytic subset of $X^\sigma$. As $Y'\setminus(\Sing(Z')\cup Z'')$ is an open $C$-semianalytic subset of $X^\sigma\cap Z$,
$$
C:=(X^\sigma\cap Z)\setminus(Y'\setminus(\Sing(Z')\cup Z''))
$$
is a closed ${\mathcal A}(X^\sigma)$-definable $C$-semianalytic subset of $X^\sigma$.

\paragraph{}We claim:
\begin{equation}\label{squnion}
S=\Int_{X^\sigma\cap Z'}(S\setminus C)\sqcup(S\cap(\Sing(Z')\cup Z'')).
\end{equation}
We only need to prove the inclusion from left to right. Let $x\in S\setminus(\Sing(Z')\cup Z'')$. We have to check that $x\in\Int_{X^\sigma\cap Z'}(S\setminus C)$.

Assume $x\in Y_1\cap V_1$. As $\dim(Y_1)=d$ and $\Sing(Y_1)\subset C$, the diference $Y_1\setminus C$ is a real analytic manifold of dimension $d$ contained in $(X^\sigma\cap Z')\setminus C$, so it is an open subset of $X^\sigma\cap Z'$. As
$$
S\setminus(\Sing(Z')\cup Z'')\subset Y'\setminus(\Sing(Z')\cup Z'')=(X^\sigma\cap Z')\setminus C,
$$
we deduce $x\in (Y_1\setminus C)\cap V_1\subset\Int_{X^\sigma\cap Z'}(S\setminus C)$.

\paragraph{}Define $Z_1=Z'$ and $U_1=X^\sigma\setminus\cl(X^\sigma\setminus(S\setminus C))$. It holds that $Z_1$ is an invariant complex analytic subset of $X$ and $U_1$ is an open ${\mathcal A}(X^\sigma)$-definable $C$-semianalytic set by \cite[Prop. 3.5]{abf1}. We have 
$$
Z_1\cap U_1=(X^\sigma\cap Z')\setminus\cl(X^\sigma\setminus(S\setminus C))=\Int_{X^\sigma\cap Z'}(S\setminus C),
$$
which is a real analytic manifold.

\paragraph{}As $S\cap(\Sing(Z')\cup Z'')$ is an ${\mathcal A}(X^\sigma)$-definable and amenable $C$-semianalytic set of dimension $<d$, there exist by induction hypothesis finitely many invariant complex analytic subsets $Z_2,\ldots,Z_r$ of $X$ and finitely many open ${\mathcal A}(X^\sigma)$-definable $C$-semianalytic subsets $U_2,\ldots,U_r$ of $X^\sigma$ such that $Z_i\cap U_i$ is a real analytic manifold, $\dim(Z_{i+1}\cap U_{i+1})<\dim(Z_i\cap U_i)$ for $i=2,\ldots,r-1$ and
$$
S\cap(\Sing(Z')\cup Z'')=\bigsqcup_{i=2}^r(Z_i\cap U_i).
$$ 
Observe that $\dim(Z_1\cap U_1)=d>\dim(S\cap(\Sing(Z')\cup Z''))=\dim(Z_2\cap U_2)$. By equation \eqref{squnion} $S=\bigsqcup_{i=1}^r(Z_i\cap U_i)$, as required.
\end{proof}

\begin{lem}\label{strata}
Let $F:(X,\an_X)\to(Y,\an_Y)$ be a proper surjective map between reduced irreducible Stein spaces of the same dimension $d$. Then, there exist complex analytic subsets $X'\subset X$ and $Y'\subset Y$ of dimension $<d$ such that:
\begin{itemize}
\item[(i)] $F^{-1}(Y')=X'$,
\item[(ii)]$M:=X\setminus X'$ and $N:=Y\setminus Y'$ are complex analytic manifolds respectively dense in $X$ and $Y$,
\item[(iii)] $F|_M:M\to N$ is an open proper surjective holomorphic map of constant rank $d$.
\end{itemize}
\end{lem}
\begin{proof}
Recall that compact analytic subsets of a Stein space are finite sets, so the fibers of $F$ are finite. By \cite[L.Theorem 4, pag. 136]{gu} the set 
$$
X_0:=\{z\in\Reg(X):\ {\rm rk}_z(F)\leq d-1\}\cup\Sing(X)
$$
is a complex analytic subset of $X$ of dimension $<d$. 

The singular set $\Sing(Y)$ has dimension $<d$. As $F$ is proper, $F(X_0)$ is by Remmert's Theorem a complex analytic subset of $Y$ of dimension $<d$. Let $Y':=\Sing(Y)\cup F(X_0)$ and $X':=F^{-1}(Y')$, which is a complex analytic subset of $X$ of dimension $<d$ because $F$ has finite fibers. Let $M:=X\setminus X_1$ and $N:=Y\setminus Y_1$, which are complex analytic manifolds respectively dense in $X$ and $Y$. As $M=F^{-1}(N)$, the map $F|_M:M\to N$ is proper and surjective. In addition, $F|_M$ has constant rank $d$, so by the rank theorem $F|_M$ is open, as required.
\end{proof}

We prove next Theorem \ref{properint-tame}.\setcounter{paragraph}{0}

\begin{proof}[Proof of Theorem \em \ref{properint-tame}]
(i) By Lemma \ref{pd} $S=\bigcup_{i=1}^rZ_i\cap V_i$ where 
\begin{itemize}
\item $Z_i$ is an invariant complex analytic subset of $X$ of (complex) dimension $d_i$ and $Z_i\cap X^\sigma$ has real dimension $d_i$,
\item $V_i$ is an open ${\mathcal A}(X^\sigma)$-definable $C$-semianalytic subset of $X^\sigma$.
\end{itemize}
As $F(S)=\bigcup_{i=1}^rF(Z_i\cap V_i)$, it is enough to prove that $F(Z_i\cap V_i)$ is an amenable $C$-semianalytic set, so we assume $S=Z_1\cap V_1$. To soften notation write $Z:=Z_1$ and $V:=V_1$. We proceed by induction on the dimension of $S$. If $S$ has dimension $0$, it is a discrete subset of $X^\sigma$ and as $F$ is proper also $F(S)$ is a discrete subset of $Y^\tau$, so it is amenable. Assume the result true if the dimension of $S$ is $<d$ and let us check that it is also true if $S$ has dimension $d$. 

Let $\{Z_\alpha\}_\alpha$ be the locally finite family of the irreducible components of $Z$. By Remmert's Theorem and Lemma \ref{proylc} $\{Y_\alpha:=F(Z_\alpha)\}_\alpha$ is a locally finite family of irreducible complex analytic subsets of $Y$. As $F(S)=\bigcup_\alpha F(V\cap Z_\alpha)$ and the family $\{Y_\alpha\}_\alpha$ is locally finite, it is enough to show by Corollary \ref{lfuok2} that $F(V\cap Z_\alpha)$ is an amenable $C$-semianalytic set. Thus, we assume $X,Y$ are irreducible, they have complex dimension $d$, $F$ is surjective and $S=V$.

By Lemma \ref{strata} there exist complex analytic subsets $X'\subset X$ and $Y'\subset Y$ of dimension $<d$ such that:
\begin{itemize}
\item $F^{-1}(Y')=X'$,
\item $M:=X\setminus X'$ and $N:=Y\setminus Y'$ are invariant complex analytic manifolds respectively dense in $X$ and $Y$,
\item $F|_M:M\to N$ is a proper open surjective holomorphic map of constant rank equal to $d$.
\end{itemize}
By induction hypothesis $F(V\cap X')$ is an amenable $C$-semianalytic set, so it is enough to prove that $F(V\cap M)$ is an amenable $C$-semianalytic set. By \cite[Thm. 1.1]{abf1} $F(V\cap M)$ is a $C$-semianalytic set. Denote $M^\sigma:=M\cap X^\sigma$ and $N^\tau:=N\cap Y^\tau$. As $F$ is invariant, $f:=F|_{M^\sigma}:M^\sigma\to N^\tau$. As ${\rm rk}_z(F)=d$ for all $z\in M$, it holds ${\rm rk}_x(f)=d$ for all $x\in M^\sigma$. As $\dim_\R(M^\sigma)=\dim_\R(N^\tau)=d$, we deduce by the rank theorem that $f$ is open, so $f(V\cap M)$ is an open $C$-semianalytic subset of $Y^\tau$, as required.

(ii) After shrinking $Y$ if necessary, we write $S'=\bigcup_{i=1}^rZ_i\cap V_i$ where $V_i$ is an open $C$-semianalytic set and $Z_i$ is an invariant complex analytic subset of $Y$. As $Y$ is Stein and $F|_{F^{-1}(Y)}:F^{-1}(Y)\to Y$ is proper, also $F^{-1}(Y)$ is by \cite[M.Thm.3, pag. 142]{gu3} Stein. We substitute $X$ by $F^{-1}(Y)$.

Write $Z_i':=F^{-1}(Z_i)$ and $W_i:=F^{-1}(V_i)\cap X^\sigma$. Observe that $Z_i'$ is an invariant complex analytic subset of $X$ and $W_i$ is an open $C$-semianalytic subset of $X^\sigma$. As $T$ is a union of connected components of $F^{-1}(S)=\bigcup_{i=1}^rZ_i'\cap W_i$, the intersection $T\cap Z_i'\cap W_i$ is a union of connected components of $Z_i'\cap W_i$. Thus, we may assume: 
\begin{itemize}
\item $S'=Z\cap V$ where $Z$ is an invariant complex analytic set and $V$ is an open $C$-semianalytic subset of $Y^\tau$.
\item $T$ is a union of connected components of $F^{-1}(S)\cap X^\sigma=Z'\cap W$ where $Z':=F^{-1}(Z)$ is an invariant complex analytic set and $W:=F^{-1}(V)\cap X^\sigma$ is an open $C$-semianalytic subset of $X^\sigma$. 
\end{itemize}

Let $\{Z'_\alpha\}_\alpha$ be the locally finite family of the irreducible components of $Z'$. By Remmert's Theorem and Lemma \ref{proylc} $\{Y_\alpha:=F(Z_\alpha')\}_\alpha$ is a locally finite family of irreducible complex analytic subsets of $Y$. As $T$ is a union of connected components of $Z'\cap W$, it holds that $T\cap Z'_\alpha$ is a union of connected components of $Z'_\alpha\cap W$. As
$$
F(T)=\bigcup_\alpha F(T\cap Z'_\alpha)
$$
and the family $\{Y_\alpha\}_\alpha$ is locally finite, it is enough to show by Corollary \ref{lfuok2} that $F(T\cap Z'_\alpha)$ is an amenable $C$-semianalytic set. Thus, we may assume that $X,Y$ are irreducible, they have complex dimension $d$ and $F$ is surjective. We have to prove that if $V$ is an open $C$-semianalytic subset of $Y^\tau$ and $T$ is a union of connected components of $F^{-1}(V)\cap X^\sigma$, then $F(T)$ is an amenable $C$-semianalytic subset of $Y^\tau$. Denote $W:=F^{-1}(V)$.

We proceed by induction on the dimension of $X$. If $X$ has dimension $0$, then $T$ is a discrete set and as $F$ is proper, $F(T)$ is also a discrete set, so it is amenable. Assume the result for dimension $<d$ and let us check that it is also true for dimension $d$. By Lemma \ref{strata} there exist complex analytic subsets $X'\subset X$ and $Y'\subset Y$ of dimension $<d$ such that:
\begin{itemize}
\item $F^{-1}(Y')=X'$,
\item $M:=X\setminus X'$ and $N:=Y\setminus Y'$ are invariant complex analytic manifolds respectively dense in $X$ and $Y$,
\item $F|_M:M\to N$ is a proper open surjective holomorphic map of constant rank equal to $d$.
\end{itemize}

As $T\cap X'$ is a union of connected components of $W\cap X'$ and $\dim(X')<d$, by induction hypothesis $F(T\cap X')$ is an amenable $C$-semianalytic set, so it is enough to prove that $F(T\cap M)$ is an amenable $C$-semianalytic set. Denote $M^\sigma:=M\cap X^\sigma$ and $N^\tau:=N\cap Y^\tau$. As $F$ is invariant, $f:=F|_{M^\sigma}:M^\sigma\to N^\tau$ and, as we have commented above, $f$ is open. As $T\cap M$ is an open subset of $M^\sigma$, we conclude that $F(T\cap M)=f(T\cap M)$ is an open subset of $Y^\tau$. It only remains to show that $F(T\cap M)$ is a $C$-semianalytic subset of $Y^\tau$.

By Lemma \ref{odes} $V=\bigcup_{j\geq1}V_j$ where $\{V_j\}_{j\geq1}$ is a locally finite family of open basic $C$-semianalytic set. Fix $j\geq1$ and observe that $T\cap F^{-1}(V_j\setminus Y')$ is a union of connected components of $F^{-1}(V_j\setminus Y')\cap X^\sigma$. Let $Y_j\subset Y$ be an invariant Stein open neighborhood of $Y^\tau$ such that $V_j$ is ${\mathcal A}(Y^\tau_j)$-definable. Then, $F^{-1}(V_j\setminus Y')\cap X^\sigma$ is ${\mathcal A}(X^\sigma_j)$-definable, where $X_j:=F^{-1}(Y_j)$ is by \cite[M. Thm. 3, pag. 141]{gu3} a Stein space. In addition the map $F_j:=F|_{X_j}:X_j\to Y_j$ is proper and surjective. As $T\cap F_j^{-1}(V_j\setminus Y')$ is a union of connected components of $F^{-1}(V_j\setminus Y')\cap X^\sigma$, we deduce by \cite[Prop. 3.5]{abf1} that $T\cap F_j^{-1}(V_j\setminus Y')$ is ${\mathcal A}(X^\sigma_j)$-definable. By \cite[Thm. 1.1]{abf1} 
$$
F_j(T\cap F_j^{-1}(V_j\setminus Y'))=F_j(T\cap M\cap F_j^{-1}(V_j))=F_j(T\cap M)\cap V_j
$$ 
is a $C$-semianalytic subset of $Y^\tau$. Thus, the locally finite union of $C$-semianalytic subsets of $Y^\tau$
$$
F(T\cap M)=\bigcup_{j\geq1}F_j(T\cap M)\cap V_j
$$
is a $C$-semianalytic subset of $Y^\tau$, as required.
\end{proof}

\begin{example}[Tameness is not preserved by proper analytic maps with finite fibers]
Let $M:=\bigsqcup_{k\geq1}\sph_k$ where $\sph_k:=\{(x,y,z)\in\R^3:\ (x-k+\frac{1}{2})^2+y^2+(z-2k)^2=4\}$, which is a locally finite union of pairwise disjoint spheres. Let 
$$
S_k:=\{(x,y,z)\in\sph_k:\ k-1\leq x\leq k,\ 0\leq y\leq\tfrac{1}{k}\}.
$$
It holds that $S:=\bigcup_{k\geq1}S_k$ is an amenable $C$-semianalytic set. Let $\pi:\R^3\to\R^2,\ (x,y,z)\mapsto(x,y)$ be the projection onto the first two variables. Observe that $\rho:=\pi|_{M}:M\to\R^2$ is a proper map with finite fibers. However 
$$
\rho(S)=\bigcup_{k\geq1}\{(x,y)\in\R^2:\ k-1\leq x\leq k,\ 0\leq y\leq\tfrac{1}{k}\}
$$
is not amenable as we have seen in Example \ref{ex}.
\end{example}

\subsection{Tameness-algorithm for $C$-semianalytic sets}\label{algorithmtame}

We end this section with an algorithm to determine if a $C$-semianalytic set is amenable. 

\subsubsection{Zariski closure of amenable $C$-semianalytic sets}

The announced algorithm is based on the fact that while the behavior of the Zariski closure of general $C$-semianalytic sets can be wild, the Zariski closure of an amenable $C$-semianalytic set behaves neatly.

\begin{example}[Bad behavior of Zariski closure of $C$-semianalytic sets]\label{ojo}
For $n\geq1$ consider the basic $C$-semianalytic set
$$
S_n:=\{y=nx,n\leq x\leq n+1\}\subset\R^2
$$
The family $\{S_n\}_{n\geq1}$ is locally finite, so $S:=\bigcup_{n\geq1}S_n$ is a $C$-semianalytic set. If $x\in S$ and $U^x$ is a small enough $C$-semianalytic neighborhood of $x$, the Zariski closure $\ol{S\cap U^x}^{\zar}$ is a line. The collection $\{\ol{S}_n^{\zar}\}$ of all these lines is not locally finite at the origin and $\ol{S}^{\zar}=\R^2$.\end{example}

\begin{lem}[Neat behavior of Zariski closures of amenable $C$-semianalytic set]\label{gbzc}
Let $S\subset M$ be an amenable $C$-semianalytic set and for each $x\in M$ let $U^x\subset M$ be any open $C$-semianalytic neighborhood. Then 
$$
\ol{S}^{\zar}=\bigcup_{x\in\cl(S)}\ol{S\cap U^x}^{\zar}.
$$
\end{lem}
\begin{proof}
By Lemma \ref{decomp} we may assume $S=X\cap W$ is a real analytic manifold where $X$ is the Zariski closure of $S$ and $W$ is an open $C$-semianalytic set. Let $\{S_i\}_{i\geq1}$ be the family of the connected components of $S$. Then
\begin{itemize}
\item $S_i$ is a pure dimensional connected amenable $C$-semianalytic set.
\item Each $\ol{S}_i^{\zar}$ is an irreducible component of $X$.
\item $S=\bigcup_{i\geq1}S_i$ and $X=\bigcup_{i\geq1}\ol{S_i}^{\zar}$.
\end{itemize}
Let $x\in\cl(S)$ and assume that $x\in\cl(S_i)$ only for $i=1,\ldots,r$. Let $C:=M\setminus\bigcup_{i\leq r+1}\cl(S_i)$ and let $V:=U^x\setminus C$, which is an open $C$-semianalytic neighborhood of $x$. Notice that $\dim(S_x)=\dim(S\cap V)=\dim(\ol{S\cap V}^{\zar})$ and
$$
\bigcup_{i=1}^r\ol{S}_i^{\zar}=\ol{S\cap V}^{\zar}\subset\ol{S\cap U^x}^{\zar}\subset X
$$
As $\cl(S)=\bigcup_{i\geq1}\cl(S_i)$, it holds $X=\bigcup_{i\geq1}\ol{S}_i^{\zar}=\bigcup_{x\in\cl(S)}\ol{S\cap U^x}^{\zar}$, as required.
\end{proof}

\subsubsection{Definition of the algorithm}
Let $S\subset M$ be a $C$-semianalytic set. Let $X_1$ be the union of the irreducible components of $\ol{S}^{\zar}$ of maximal dimension. Define 
$$
T_1(S):=\Int_{\Reg(X_1)}(S\cap\Reg(X_1)). 
$$
Assume we have already defined $T_1(S),\ldots,T_k(S)$ and let us define $T_{k+1}(S)$. Let $R_{k+1}:=S\setminus\bigcup_{j=1}^kT_j(S)$ and let $X_{k+1}$ be the union of the irreducible components of $\ol{R_{k+1}}^{\zar}$ of maximal dimension. Consider 
$$
T_{k+1}(S):=\Int_{\Reg(X_{k+1})}(S\cap\Reg(X_{k+1})).
$$

\begin{thm}\label{algor1}
Let $S\subset M$ be a $C$-semianalytic set. The following assertions are equivalent
\begin{itemize}
\item[(i)] $S$ is amenable.
\item[(ii)] $S=\bigcup_{j\geq1}T_j(S)$.
\end{itemize}
\end{thm}

\subsubsection{Preliminary properties of the operators $T_k(\cdot)$}\label{clear}
Let $S\subset M$ be a $C$-semianalytic set.

(i) \em $T_k(S)$ is an amenable $C$-semianalytic set \em by Lemma \ref{openX}.

(ii) \em If $T_k(S)=\varnothing$, then $R_{k+\ell}=R_k$ for all $\ell\geq1$ and $T_{k+\ell}(S)=T_k(S)=\varnothing$ for all $\ell\geq1$\em.

(iii) \em If $T_{k+1}(S)\subset T_k(S)$, then $R_{k+\ell}=R_{k+1}$ for all $\ell\geq1$ and $T_{k+\ell}(S)=T_{k+1}(S)$ for all $\ell\geq1$\em.

(iv) \em If $T_{k+1}(S)\neq\varnothing$, then $\dim(T_{k+1}(S))=\dim(X_{k+1})\leq\dim(X_k)=\dim(T_k(S))$ \em because $R_{k+1}\subset R_k$.

(v) \em If $\dim(X_{k+1})=\dim(X_k)$, then $\varnothing\neq T_{k+1}(S)\setminus\Sing(X_k)\subset T_k(S)$ and $T_{k+2}(S)=T_{k+1}(S)$\em.

(vi) \em There exists $k_0\geq1$ such that $T_{k_0}(S)=T_{k_0+\ell}(S)$ for all $\ell\geq1$\em.

\begin{proof}
Assertions (i) to (iv) are straightforwardly checked.

(v) As $\dim(X_{k+1})=\dim(X_k)$, then $\dim(\ol{R_{k+1}}^{\zar})=\dim(\ol{R_k}^{\zar})$. As $R_{k+1}\subset R_k$, it holds that $\ol{R_{k+1}}^{\zar}\subset\ol{R_k}^{\zar}$, so $X_{k+1}\subset X_k$ and $\Reg(X_{k+1})\setminus\Sing(X_k)$ is a non-empty open subset of $\Reg(X_k)$ because $\dim(\Sing(X_k))<\dim(X_{k+1})$. Thus $T_{k+1}(S)\setminus\Sing(X_k)\subset T_k(S)$.

We claim: $R_{k+1}\setminus\Sing(X_k)\subset R_{k+2}$. 

As $T_{k+1}(S)\setminus\Sing(X_k)\subset T_k(S)$, we have $R_{k+1}\setminus((T_{k+1}(S)\setminus\Sing(X_k))=R_{k+1}$ and
\begin{equation*}
\begin{split}
R_{k+2}=R_{k+1}\setminus T_{k+1}(S)&=R_{k+1}\setminus((T_{k+1}(S)\setminus\Sing(X_k))\cup(T_{k+1}(S)\cap\Sing(X_k)))\\
&=(R_{k+1}\setminus(T_{k+1}(S)\setminus\Sing(X_k)))\setminus(T_{k+1}(S)\cap\Sing(X_k))\\
&=R_{k+1}\setminus(T_{k+1}(S)\cap\Sing(X_k))\supset R_{k+1}\setminus\Sing(X_k).
\end{split}
\end{equation*}

As $\dim(\Sing(X_k))<\dim(X_{k+1})$ and $R_{k+1}\setminus\Sing(X_k)\subset R_{k+2}\subset R_{k+1}$, we have 
$$
X_{k+1}\subset\ol{R_{k+1}\setminus\Sing(X_k)}^{\zar}\subset\ol{R_{k+2}}^{\zar}\subset\ol{R_{k+1}}^{\zar}.
$$
Thus, $X_{k+1}=X_{k+2}$, so $T_{k+2}(S)=T_{k+1}(S)$.

(vi) Let $k\geq1$. If $T_k(S)=\varnothing$, then by (ii) $T_k(S)=T_{k+\ell}(S)$ for all $\ell\geq1$. Thus, we assume $T_k(S)\neq\varnothing$. We know by (iv) that $\dim(X_{k+1})\leq\dim(X_k)$. If $\dim(X_{k+1})=\dim(X_k)$, then by (v) $T_{k+2}(S)=T_{k+1}(S)$, so by (iii) $T_{k+\ell}(S)=T_{k+1}(S)$ for all $\ell\geq1$. The case missing is $\dim(X_{k+1})<\dim(X_k)$. Repeat the previous argument with the index $k+1$ in the place of $k$ and observe that in finitely many steps we achieve the statement.
\end{proof}

\subsubsection{Proof of Theorem \em \ref{algor1}}
By \ref{clear} it follows that (ii) $\Longrightarrow$ (i). To prove the converse we proceed as follows. 

\noindent{\bf Step 1.} We claim: \em For each $k\geq1$ it holds $S=\bigcup_{j=1}^kT_j(S)\cup S'_{k+1}$ where $S'_{k+1}$ is either empty or an amenable $C$-semianalytic set such that $\ol{S'_{k+1}}^{\zar}=\ol{R_{k+1}}^{\zar}$ and $\dim(S'_{k+1})<\dim(T_k(S))$.\em

Let $k\geq 1$ and assume $S=L_k\cup S'_k$ where 
$$
L_k:=\begin{cases}
\varnothing&\text{if $k=1$,}\\
\bigcup_{j=1}^{k-1}T_j(S)&\text{if $k\geq2$},
\end{cases}
$$
$R_1:=S'_1:=S$ and $S'_k$ is either empty or an amenable $C$-semianalytic set such that $\ol{S'_k}^{\zar}=\ol{R_k}^{\zar}$ and $\dim(S_k')<\dim(T_{k-1}(S))$ if $k\geq2$. Let us prove that $S=L_{k+1}\cup S'_{k+1}$ where $L_{k+1}$ and $S'_{k+1}$ satisfy the corresponding properties.

If $S_k'=\varnothing$, then $T_k(S)=\varnothing$ and it is enough to take $S_{k+1}'=\varnothing$. Assume $S_k'\neq\varnothing$, so it is an amenable $C$-semianalytic set such that $\ol{S'_k}^{\zar}=\ol{R_k}^{\zar}$ and $\dim(S_k')<\dim(T_{k-1}(S))$.

Notice that $S_k'\cap\Reg(X_k)\subset S\cap\Reg(X_k)$ has no empty interior in $\Reg(X_k)$, so $T_k(S)\neq\varnothing$. Write $S_k'=\bigsqcup_{i=1}^r(Z_i\cap U_i)$ where $Z_i$ is a $C$-analytic set, $U_i$ is an open $C$-semianalytic set, each $Z_i\cap U_i$ is a real analytic manifold, $Z_i$ is the Zariski closure of $Z_i\cap U_i$ and $\dim(Z_{i+1}\cap U_{i+1})<\dim(Z_i\cap U_i)$ for $i=1,\ldots,r-1$ (see Lemma \ref{decomp}).

Let $S_{k+1}':=S_k'\cap\bigcup_{i=1}^r\ol{Z_i\cap U_i\cap R_{k+1}}^{\zar}$, which is amenable because it is the intersection of two amenable $C$-semianalytic sets. In addition $S_k'\cap R_{k+1}\subset S_{k+1}'$, so
$$
S=\bigcup_{j=1}^{k-1}T_j(S)\cup S_k'=\bigcup_{j=1}^kT_j(S)\cup (S_k'\cap R_{k+1})=\bigcup_{j=1}^{k}T_j(S)\cup S_{k+1}'=L_{k+1}\cup S_{k+1}'. 
$$
We have to prove that $\ol{S_{k+1}'}^{\zar}=\ol{R_{k+1}}^{\zar}$. 

For each $i=1,\ldots,r$ we have $\ol{Z_i\cap U_i\cap R_{k+1}}^{\zar}\subset\ol{R_{k+1}}^{\zar}$. Thus $S_{k+1}'\subset\ol{R_{k+1}}^{\zar}$, so $\ol{S_{k+1}'}^{\zar}\subset\ol{R_{k+1}}^{\zar}$. The converse inclusion follows because $R_{k+1}=S\setminus L_{k+1}\subset S_{k+1}'$.

We claim: \em $\dim(\ol{Z_i\cap U_i\cap R_{k+1}}^{\zar})<\dim(T_k(S))$ for $i=1,\ldots,r$\em. For $i=2,\ldots,r$ the result is clear, so let us prove $\dim(\ol{Z_1\cap U_1\cap R_{k+1}}^{\zar})<\dim(T_k(S))$.

As $\dim(Z_1\cap U_1)=\dim(S_k')=\dim(T_k(S))$, we have $(Z_1\cap U_1)\setminus\Sing(X_k)$ is a real analytic manifold of dimension $\dim(X_k)$, so it is an open subset of $\Reg(X_k)$ and $(Z_1\cap U_1)\setminus\Sing(X_k)\subset T_k(S)$. Thus, $(Z_1\cap U_1)\setminus T_k(S)\subset\Sing(X_k)$, so 
$$
\dim(\ol{(Z_1\cap U_1)\setminus T_k(S)}^{\zar})\leq\dim(\Sing(X_k))<\dim(X_k)=\dim(S)=\dim(T_k(S)).
$$
We conclude $\dim(S_{k+1}')<\dim(T_k(S))$.

\noindent{\bf Step 2.} To finish we prove that there exists an index $k\geq1$ such that $S_k'=\varnothing$. For each $k\geq1$, it holds that either $S_k'=\varnothing$ or
$$
\dim(X_k)=\dim(\ol{R_k^{\zar}})=\dim(\ol{S_k'}^{\zar})=\dim(S_k')<\dim(T_{k-1}(S))=\dim(X_{k-1}).
$$
As $\dim(X_1)=\dim(S)<+\infty$, there exist only finitely many posible $k$'s such that $S_k'\neq\varnothing$. Consequently, there exists $k\geq1$ such that $S=\bigcup_{j=1}^kT_j(S)=\bigcup_{j\geq1}T_j(S)$, as required.
\qed

\begin{example}\label{dimktamenot}
Consider the open $C$-semianalytic set
$$
S_0:=\bigcup_{k\geq1}\{0<x<k, 0<y<1/k\}\subset\R^3
$$
and let $S_1:=\{z=0\}$. Define $S:=S_0\cup S_1$ which is amenable because it is the union of two amenable $C$-semianalytic sets. Let $S_{(3)}$ be the set of points of $S$ of local dimension $3$ and let us check that it is not amenable. Indeed, observe first that $S_{(3)}=S_0\cup\cl(S_0\cap\{z=0\})$. Assume by way of contradiction that $S_{(3)}$ is amenable. Then also $S':=S_{(3)}\cap\{z=0\}=\cl(S_0\cap\{z=0\})$ should be amenable. To prove that $S'$ is not amenable we apply Lemma \ref{algor1}. Observe that $X_1=\{z=0\}$, $T_1(S')=S_0\cap\{z=0\}$, $R_2=\cl(S_0\cap\{z=0\})\setminus(S_0\cap\{z=0\})$, $X_2=\{z=0\}$, $T_2(S')=S_0\cap\{z=0\}$ and $S'\neq T_1(S')$. By Theorem \ref{algor1} $S'$ is not amenable, so $S_{(3)}$ is not amenable.
\end{example}

\section{Irreducible amenable $C$-semianalytic sets}\label{s4}

An amenable $C$-semianalytic subset $S$ of a real analytic manifold $M$ is \em irreducible \em if the ring $\an(S)$ is an integral domain. This notion extends both the concepts of irreducible $C$-analytic set and irreducible semialgebraic set \cite[3.2]{fg}.

\subsection{Basic properties concerning irreducibility}\label{bpci}
One deduces straightforwardly the following facts concerning irreducibility: 

(i) Irreducible amenable $C$-semianalytic sets are connected because the ring of analytic functions of a disconnected amenable $C$-semianalytic set is the direct sum of the rings of analytic functions of its connected components, so it contains zero divisors. In particular, a real analytic manifold is irreducible if and only if it is connected. 

(ii) The Zariski closure $\ol{S}^{\zar}_U$ of an irreducible amenable $C$-semianalytic set $S$ in an open neighborhood $U$ is irreducible because $\an(\ol{S}^{\zar}_U)\hookrightarrow\an(S), f\mapsto f|_S$. As $S$ is amenable, $\dim(\ol{S}^{\zar}_U)=\dim(S)$.

(iii) An amenable $C$-semianalytic set that is the image of an irreducible amenable $C$-semi\-analytic set under an analytic map is irreducible. In particular, the irreducibility of amenable $C$-semianalytic sets is preserved under analytic diffeomorphisms.

(iv) Let $T\subset S\subset\R^n$ be amenable $C$-semianalytic sets such that $T$ is irreducible. Then the ideal
$$
\ideal(T,S):=\{f\in\an(S):\ f|_T=0\}
$$ 
is a prime ideal of $\an(S)$ because $\an(T)$ is an integral domain and $\ideal(T,S)$ is the kernel of the restriction homomorphism $\an(S)\to\an(T),\ f\mapsto f|_T$.

\begin{example}
The previous notion has a misleading behavior if we extend it to arbitrary $C$-semianalytic sets. Let $U\subset\R^3$ be an open connected neighborhood of $S:=\bigcup_{k\geq0}S_k$ where 
$$
S_k:=\{z^2-(y-k)x^2=0,y\geq k\}
$$ 
(Example \ref{tamenot}). If $f\in\an(U)$ vanishes identically on $S$, then $f=0$ on $U$, so $\ideal(S)=(0)$ and $\ol{S}^{\zar}_U=U$. Consequently, $\an(S)=H^0(S,\an_M|_S)$ is an integral domain because $S$ is connected. Thus, $S$ should be irreducible, even if we feel that it should not. Note that $S^{(\ell)}:=\bigcup_{k=0}^\ell S_k$ is reducible for all $\ell$. Its irreducible components are $\{z^2-(y-k)x^2=0,y\geq 0\}$ for $0\leq k\leq\ell$.
\end{example}

\begin{lem}\label{charirred}
Let $S$ be an amenable $C$-semianalytic set. The following assertions are equivalent:
\begin{itemize}
\item[(i)] $S$ is irreducible.
\item[(ii)] For each open neighborhood $U$ of $S$ in $M$, the $C$-analytic set $\ol{S}^{\zar}_U$ is irreducible. 
\item[(iii)] If $f\in\an(S)$ and $\dim(\ceros(f))=\dim(S)$, then $f$ is identically zero.
\end{itemize}
\end{lem}
\begin{proof}
(i) $\Longrightarrow$ (ii) The ideal $\ideal(S,U)$ is a prime ideal of $\an(U)$, so $\ol{S}^{\zar}_U=\ceros(\ideal(S,U))$ is an irreducible $C$-analytic subset of $U$.

(ii) $\Longrightarrow$ (iii) Let $f'\in\an(U)$ be an analytic extension of $f$ to an open neighborhood $U\subset M$ of $S$. As $f'$ vanishes on a subset of maximal dimension of the irreducible $C$-analytic set $\ol{S}^{\zar}_U$, we have $f'|_{\ol{S}^{\zar}_U}\equiv 0$, so $f=f'|_S\equiv0$.

(iii) $\Longrightarrow$ (i) Let $f_1, f_2\in\an(S)$ be such that $f_1f_2\equiv 0$. Let $x\in S$ be such that the germ $S_x$ is regular of maximal dimension. As $S_x$ is irreducible, we may assume $f_1$ is identically zero on an open neighborhood of $x$ in $S$. Thus, $\dim(\ceros(f_1))=\dim(S)$, so $f_1\equiv0$. Consequently, $S$ is irreducible.
\end{proof}

\subsection{Normalization and irreducibility}
We are ready to prove Theorem \ref{dpm} that relates irreducibility with connexion in the normalization of the complexification of the Zariski closure.

\begin{proof}[Proof of Theorem \em \ref{dpm}]\setcounter{paragraph}{0} 
It is enough to prove: \em For each open neighborhood $U\subset M$ of $S$ the Zariski closure $X:=\ol{S}^{\zar}_U$ is irreducible if and only there exists a connected component $T$ of $\pi^{-1}(S)$ such that $\pi(T)=S$\em.

\paragraph{} Suppose first $\pi(T)=S$ for some connected component $T$ of $\pi^{-1}(S)$. Assume $M$ is imbedded in $\R^n$ as a closed real analytic submanifold and $\widetilde{X}$ is a complex analytic subset of $\C^n$. Note that $\dim_\C(\widetilde{X})=\dim_\R(X)=:d$. Fix an open neighborhood $U\subset M$ of $S$ and let $\Omega\subset\C^n$ be an invariant open neighborhood of $S$ such that $\Omega\cap M=U$. By Remark \ref{norm1} $(\pi^{-1}(\widetilde{X}\cap\Omega),\pi|)$ is the normalization of $\widetilde{X}\cap\Omega$. As $T$ is connected and $\dim_\R(T)=d$, it is contained in a connected component $Z$ of $\pi^{-1}(\widetilde{X}\cap\Omega)$ of dimension $d$. By Remark \ref{norm1} $\pi(Z)$ is an irreducible component of $\widetilde{X}\cap\Omega$ of dimension $d$. As $S=\pi(T)\subset\pi(Z)\subset \widetilde{X}$, we have $\ol{S}^{\zar}_U\subset\pi(Z)\cap U$. A $\pi(Z)$ is irreducible and has dimension $d$, it is the complex Zariski closure of $\ol{S}^{\zar}_U$ in $\Omega$. As this holds for all invariant open neighborhood $\Omega\subset\C^n$ of $\ol{S}^{\zar}_U$, we deduce $\ol{S}^{\zar}_U$ is irreducible.

\paragraph{}\label{pieces} Suppose next that $\ol{S}^{\zar}_U$ is irreducible for each open neighborhood $U\subset M$ of $S$. We will construct a suitable open neighborhood $U\subset M$ of $S$ in $M$ to prove the existence of a connected component $T$ of $\pi^{-1}(S)$ such that $\pi(T)=S$. 

Let $\{T_i\}_{i\geq1}$ be the connected components of $\pi^{-1}(S)$. As $S$ is invariant, $\pi^{-1}(S)$ is invariant. Let $\{\Theta_i'\}_{i\geq1}$ be pairwise disjoint open subsets of $Y$ such that $T_i\subset\Theta_i'$ for $i\geq1$. Denote $\Theta':=\bigcup_{i\geq1}\Theta_i'$ and $\Theta:=\Theta'\cap\widehat{\sigma}(\Theta')$. Notice that $\Theta$ is an invariant neighborhood of $\pi^{-1}(S)$ in $Y$ and $\Theta_i:=\Theta_i'\cap\Theta\subset\widetilde{X}$ is an open neighborhood of $T_i$. Clearly, $\Theta_i\cap\Theta_j=\varnothing$ if $i\neq j$.

Define $C:=Y\setminus\Theta$ which is a closed invariant subset of $Y$ that does not intersect $\pi^{-1}(S)$. As $\pi$ is proper and invariant, $\pi(C)$ is an invariant closed subset of $\widetilde{X}$. It holds $S\cap\pi(C)=\varnothing$, so $\pi^{-1}(S)\cap\pi^{-1}(\pi(C))=\varnothing$. Substituting $C$ by the invariant closed set $\pi^{-1}(\pi(C))$, we may assume that $C=\pi^{-1}(\pi(C))$, so the restriction map $\pi|_{Y\setminus C}:Y\setminus C\to \widetilde{X}\setminus\pi(C)$ is proper and surjective. Define $\Omega:=\C^n\setminus\pi(C)$ and $U:=\Omega\cap M$, which is an open neighborhood of $S$ in $M$.

\paragraph{}As $\ol{S}^{\zar}_U$ is irreducible, the complex Zariski closure $Z$ of $\ol{S}^{\zar}_U$ in $\Omega$ is irreducible, it is contained in $\widetilde{X}$ and its dimension equals $\dim_\R(S)=\dim_\R(X)=\dim_\C(\widetilde{X})$. Thus, it is an irreducible component of $\widetilde{X}':=\widetilde{X}\cap\Omega$. In addition $(Y':=\pi^{-1}(\widetilde{X}'),\pi|_{Y'})$ is by Remark \ref{norm1} the normalization of $\widetilde{X}'$. Thus, there exists a connected component $K$ of $Y'$ such that $Z=\pi(K)$.

As $K\subset\bigcup_{i\geq1}\Theta_i$ is connected, $K\subset\Theta_{i_0}$ for some $i_0\geq1$. As $T_i\cap K\subset T_i\cap\Theta_{i_0}=\varnothing$ if $i\neq i_0$,
$$
\pi(T_{i_0})=\pi\Big(K\cap\bigcup_{i\geq1}T_i\Big)=\pi(K\cap\pi^{-1}(S))=Z\cap S=S,
$$
as required. 
\end{proof}

\begin{remarks}
(i) Let $Y^{\widehat{\sigma}}$ be the set of fixed points of $\widehat{\sigma}$. Recall that if $X$ is coherent, $\pi^{-1}(X)=Y^{\widehat{\sigma}}$, see \cite[Thm. 3.14]{gmt}. If such is the case, the connected component $T$ in the statement of Theorem \ref{dpm} is contained in $Y^{\widehat{\sigma}}$.

(ii) If $X$ is not coherent but $\pi(\pi^{-1}(S)\cap Y^{\widehat{\sigma}})=S$, we cannot assure that the connected component $T$ of $\pi^{-1}(S)$ such that $\pi(T)=S$ satisfies in addition $\pi(T\cap Y^{\widehat{\sigma}})=S$. Let $X$ be the irreducible $C$-analytic set of equation $x^4-z^2(4-z^2)y^4=0$. Consider the $C$-semianalytic set 
$$
S:=(X\cap\{0<z<1,y\neq0\})\cup\{x=0,y=0,-1<z<1\}.
$$
\begin{center}
\begin{figure}[ht]
\hspace*{-3mm}\includegraphics[width=.65\textwidth]{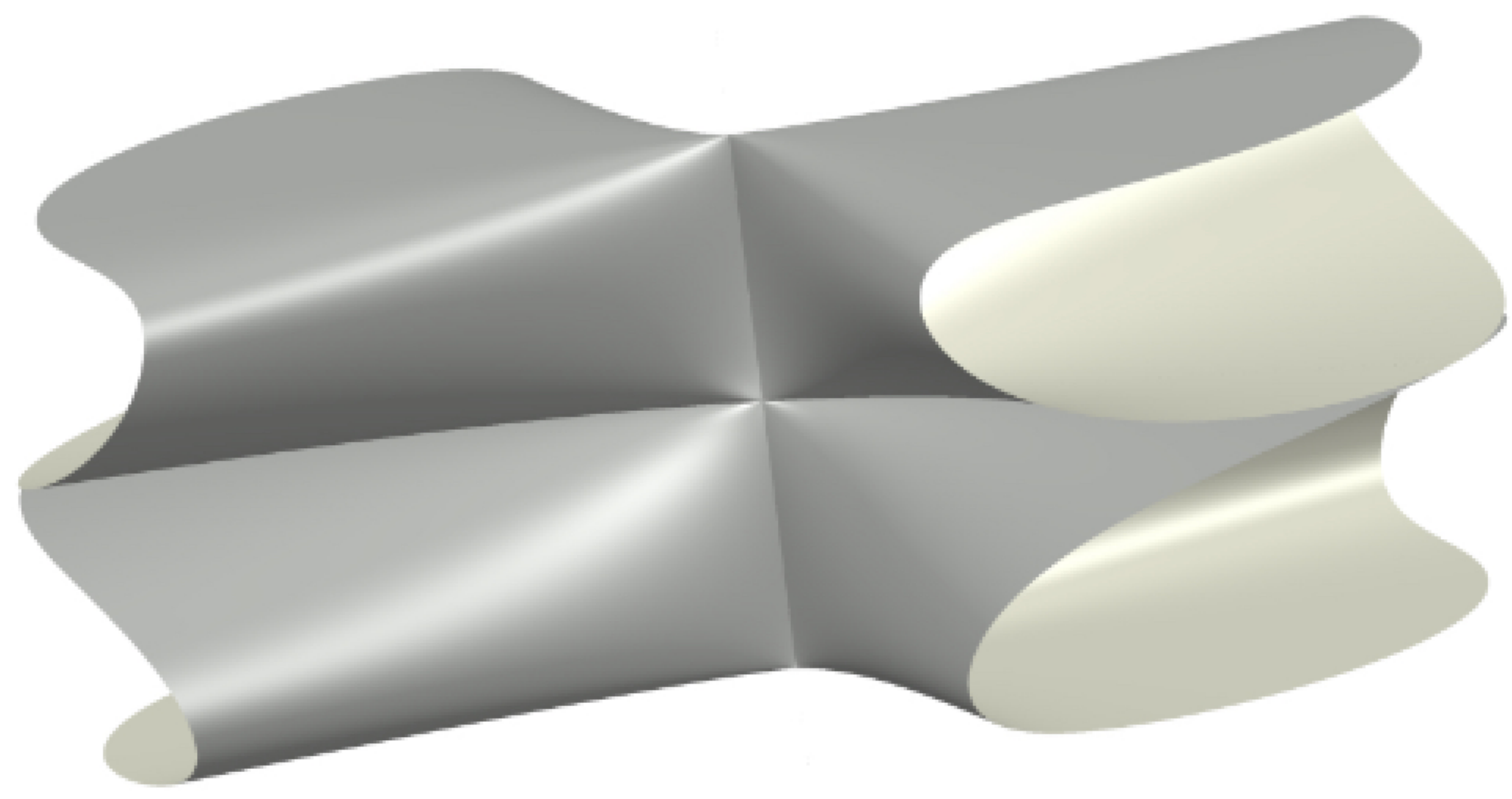}
\caption{$X:=\{x^4-z^2(4-z^2)y^4=0\}$}
\end{figure}
\end{center}
\vspace*{-7mm}
Consider the analytic diffeomorphism
$$
f:\R^3\to\R^2\times(-1,1),\ (x,y,z)\mapsto\Big(x,y,\frac{z}{\sqrt{1+z^2}}\Big).
$$
Observe that 
\begin{multline*}
f^{-1}(X\cap(\R^2\times(-1,1)))=\{x^4(1+z^2)-z^2(4+3z^2)y^4=0\}\\
=\big\{(x\sqrt[4]{1+z^4})^2-z(\sqrt[4]{4+3z^2})=0\big\}\cup\big\{(x\sqrt[4]{1+z^4})^2+z(\sqrt[4]{4+3z^2})=0\big\}.
\end{multline*} 
and
$$
f^{-1}(S)=\big\{(x\sqrt[4]{1+z^4})^2-z(\sqrt[4]{4+3z^2})=0\big\}
$$ 
is analytically equivalent to Whitney's umbrella, so it is irreducible. The complex analytic set $\widetilde{X}:=\{x^4-z^2(4-z^2)y^4=0\}\subset\C^3$ is a complexification of $X$. Its normalization is the non-singular complex analytic set $Y:=\{u^2-vz=0,v^2+z^2-4=0\}\subset\C^4$ together with the holomorphic map $\pi:Y\to\widetilde{X},\ (y,z,u,v)\mapsto(uy,y,z)$. We have $\pi^{-1}(S)=T_1\cup T_2\cup T_3\cup T_4$ where
\begin{align*}
T_1&:=\Big\{(y,z,\pm\sqrt{z\sqrt{4-z^2}},\sqrt{4-z^2}), 0\leq z<1\Big\}\\
T_2&:=\Big\{(0,z,\pm\sqrt{-z\sqrt{4-z^2}}\sqrt{-1},\sqrt{4-z^2}),-1<z\leq0\Big\}\\
T_3&:=\Big\{(0,z,\pm\sqrt{z\sqrt{4-z^2}}\sqrt{-1},-\sqrt{4-z^2}),0\leq z<1\Big\}\\
T_4&:=\Big\{(0,z,\pm\sqrt{-z\sqrt{4-z^2}},-\sqrt{4-z^2}),-1<z\leq0\Big\}
\end{align*}
The connected components of $\pi^{-1}(S)$ are $T:=T_1\cup T_2$ and $T':=T_3\cup T_4$. It holds $\pi(T)=S$. In addition $\pi^{-1}(S)\cap Y^{\widehat{\sigma}}=T_1\cup T_4$ and it satisfies $\pi(T_1\cup T_4)=S$. However, $T_1\cup T_4$ is not connected and $\pi(T\cap Y^{\widehat{\sigma}})=\pi(T_1)\subsetneq S$.
\end{remarks}

\section{Irreducible components of an amenable $C$-semianalytic set}\label{s5}

The next natural step is to explore the notion of \em irreducible components \em of an amenable $C$-semianalytic set (see Definition \ref{irredcomptame} and Theorem \ref{bijection}). This theory shall generalize theories of irreducible components for particular cases of amenable $C$-semianalytic sets as: $C$-analytic sets, complex analytic subset of a Stein manifold (endowed with their underlying real structure) and semialgebraic sets. More precisely,

(i) If $X$ is a $C$-analytic set, its irreducible components $\{X_i\}_{i\geq1}$ as a $C$-analytic set coincides with the ones obtained if we consider $X$ as an amenable $C$-semianalytic set. 

\begin{proof}
Let us check that $\{X_i\}_{i\geq1}$ satisfies the conditions in Definition \ref{irredcomptame}. Only condition (2) requires a comment. Let $X_1\subset T\subset X$ be an irreducible amenable $C$-semianalytic set. By Lemma \ref{ni} below there exists $j\geq1$ such that $X_1\subset T\subset X_j$, so $j=1$ and $T=X_1$, as required.
\end{proof}

(ii) Recall that if $X$ is an irreducible complex analytic subset of a Stein manifold \cite[IV.\S1.Cor.1, pag.68]{n}, then $\Reg(X):=X\setminus\Sing(X)$ is connected. If we consider the real structure $(X^\R,\an_X^\R)$ induced by $X$, we deduce that $X^\R$ is irreducible as a $C$-analytic set. In addition if $X$ is a general complex analytic subset of a Stein manifold, the irreducible components of $X$ are by \cite[IV.\S1.Cor.2, pag.67]{n} the closures in $X$ of the connected components of $\Reg(X)$. Consequently the irreducible components of $X$ and the irreducible components of $X^\R$ as a $C$-analytic set coincide. Thus, the irreducible components of $X$ as a complex analytic set coincides with the ones obtained if we consider $X$ as an amenable $C$-semianalytic set. 

(iii) If $S$ is a semialgebraic set, its irreducible components as a semialgebraic set coincides with the ones obtained if we consider $X$ as an amenable $C$-semianalytic set \cite[3.2, 4.1]{fg}. 

We prove next the existence and the unicity of the family of irreducible components of an amenable $C$-semianalytic set in the sense of Definition \ref{irredcomptame}. To get advantage of the full strength of the algebraic properties of the ring $\an(S)$ we introduce first some milder concepts and study their main properties. 

\subsection{Weak irreducible components of an amenable $C$-semianalytic set}\label{irredcomp}

Let $S\subset M$ be a subset. We say that $T\subset S$ is an \em $S$-tame $C$-semianalytic set \em if there exists an open neighborhood $U\subset M$ of $S$ such that $T$ is an amenable $C$-semianalytic subset of $U$. There is no ambiguity to say that an $S$-tame $C$-semianalytic set $T\subset S$ is \em irreducible \em if $\an(T)$ is an integral domain. Given an amenable $C$-semianalytic set $S\subset M$ and an $S$-tame $C$-semianalytic set $T\subset S$, we define the \em ideal of $T$ with respect to $S$ \em as
$$
\ideal(T,S):=\{f\in\an(S):\ f|_T\equiv0\}.
$$

\begin{defn}[Weak irreducible components]\label{irredcompwtame}
Let $S\subset M$ be an amenable $C$-semi\-analytic set. A countable locally finite family $\{S_i\}_{i\geq1}$ in $S$ of $S$-tame $C$-semianalytic sets is \em a family of weak irreducible components of \em $S$ if the following conditions are fulfilled: 
\begin{itemize}
\item[(1)] Each $S_i$ is irreducible.
\item[(2)] If $S_i\subset T\subset S$ is an irreducible $S$-tame $C$-semianalytic set, then $S_i=T$.
\item[(3)] $S_i\neq S_j$ if $i\neq j$.
\item[(4)] $S=\bigcup_{i\geq1} S_i$.
\end{itemize} 
\end{defn}

\begin{thm}[Existence and uniqueness of weak irreducible components]\label{irredcomp2}
Let $S\subset\R^n$ be an amenable $C$-semianalytic set. Then there exists the family of weak irreducible components $\{S_i\}_{i\geq1}$ of $S$ and it is unique. In addition it satisfies 
\begin{itemize}
\item[(i)] $S_i=\ceros(\ideal(S_i,S))$ for $i\geq1$. In particular, $S_i$ is a closed subset of $S$.
\item[(ii)] The ideals $\ideal(S_i,S)$ are the minimal prime (saturated) ideals of $\an(S)$.
\end{itemize}
\end{thm}

\subsubsection{Preliminary results}
Before proving Theorem \ref{irredcomp2} we introduce some auxiliary results.

\begin{lem}\label{sheaf}
Let $S\subset M$ be a semianalytic set and let $\gta$ be an ideal of $\an(S)$. Then $\ceros(\gta)$ is the intersection of $S$ with a $C$-analytic subset $X$ of an open neighborhood $U\subset M$ of $S$. 
\end{lem}
\begin{proof}
We have defined $\an(S)$ in \eqref{os} as the quotient $H^0(S,\an_M|_S)/\ideal(S)$. Let $\gtA\supset\ideal(S)$ be the ideal of $H^0(S,\an_M|_S)$ such that $\gta=\gtA/\ideal(S)$. By \cite[Cor. (I,8)]{f} the sheaf of ideals $\gtA\an_M|_S$ is $\an_M|_S$-coherent. By \cite[I.2.8]{gmt} \em there exists an open neighborhood $U\subset M$ of $S$ and an analytic $\an_U$-coherent sheaf ${\mathcal F}$ such that $\gtA\an_M|_S={\mathcal F}|_S$.\em

Let $X$ be the zero set of the $\an_U$-coherent sheaf ${\mathcal F}$, which is a $C$-analytic subset of $U$. As $\gtA\an_M|_S={\mathcal F}|_S$, we deduce $\ceros(\gta)=\ceros(\gtA)=S\cap X$, as required.
\end{proof}

\begin{lem}\label{clue2}
Let $S\subset M$ be a semianalytic set and let $\gta$ be an ideal of $\an(S)$. Denote $T=\ceros(\gta)$. Then for each $f\in\an(T)$ there exists $g\in\an(S)$ such that $\ceros(f)=\ceros(g)$.
\end{lem}
\begin{proof}
By Lemma \ref{sheaf} there exists an open neighborhood $U_0\subset M$ of $S$ and a $C$-analytic set $X\subset U_0$ such that $T=S\cap X$. Let $V\subset U_0$ be an open neighborhood of $T$ and ${f'}\in\an(V)$ an analytic function such that ${f'}|_T=f$. We claim: $S\subset U_1:=(U\setminus X)\cup V$. 

Indeed, $S\cap U_1=(S\setminus X)\cup(V\cap S)\supset(S\setminus T)\cup T=S$, so $S\subset U_1$. 

It holds: \em $X':=X\cap U_1=X\cap V$ is a $C$-analytic subset of $U_1$\em. 

Let $g$ be an analytic equation of $X'$ in $V$ and consider the $\an_{U_1}$-coherent sheaf of ideals
$$
{\mathcal F}_x:=\begin{cases}
g\an_{U_1,x}&\text{ if $x\in V$},\\
\an_{U_1,x}&\text{ otherwise}.
\end{cases}
$$
As $X'$ is the zero set of ${\mathcal F}$, we conclude that $X'$ is a $C$-analytic subset of $U_1$.

In addition $S\cap X'=S\cap X\cap V=T$. Let $h\in\an(U_1)$ be an analytic equation of $X'$ in $U_1$ and consider the $\an_{U_1}$-coherent sheaf of ideals
$$
{\mathcal F}_x:=\begin{cases}
(h^2+{f'}^2)\an_{U_1,x}&\text{ if $x\in X'$},\\
\an_{U_1,x}&\text{ otherwise}.
\end{cases}
$$
Its zero set is a $C$-analytic subset of $U_1$, so there exists $g'\in\an(U_1)$ such that its zero set coincides with the zero set $\ceros(h,{f'})=X'\cap\ceros({f'})$ of ${\mathcal F}$. Thus, if $g:=g'|_S$, we have
$$
\ceros(g)=S\cap\ceros(g')=S\cap X'\cap\ceros({f'})=T\cap\ceros({f'})=\ceros(f),
$$
as required.
\end{proof}

\begin{lem}\label{clue3}
Let $S\subset M$ be an amenable $C$-semianalytic set and let $\gta$ be an ideal of $\an(S)$. Then 
\begin{itemize}
\item[(i)] $\ceros(\gta)$ is an $S$-tame $C$-semianalytic set.
\item[(ii)] $\ceros(\gta)$ is irreducible if and only if $\ideal(\ceros(\gta),S)$ is a prime ideal of $\an(S)$. 
\end{itemize}
\end{lem}
\begin{proof}
(i) This statement follows from Lemma \ref{sheaf}.

(ii) Recall that if $T=\ceros(\gta)$ is irreducible, then $\ideal(T,S)$ is prime (see \ref{bpci}(iv)). Conversely, assume that $\ideal(T,S)$ is prime and let $f_1,f_2\in\an(T)$ be such that $f_1f_2=0$. By \ref{clue2} there exist analytic functions $g_1,g_2\in\an(S)$ such that $\ceros(g_i)=\ceros(f_i)$ for $i=1,2$. Thus, $\ceros(g_1g_2)=\ceros(f_1f_2)=T$, so $g_1g_2\in \ideal(T,S)$. As $\ideal(T,S)$ is a prime ideal, we assume $g_1\in\ideal(T,S)$. Thus, $\ceros(f_1)=\ceros(g_1)=T$, so $f_1=0$. Consequently, $\an(T)$ is an integral domain and $T$ is irreducible.
\end{proof}

\begin{lem}\label{ni}
Let $S\subset T\subset E\subset M$ be $E$-tame $C$-semianalytic sets such that $S$ is irreducible and let $\{T_i\}_{i\geq1}$ be a family of $E$-tame $C$-semianalytic sets. Assume 
\begin{itemize}
\item[(i)] $T_i=\ceros(\ideal(T_i,E))$ for $i\geq1$,
\item[(ii)] $T=\bigcup_{i\geq1}T_i$,
\item[(iii)] The family $\{T_i\}_{i\geq1}$ is locally finite in $E$,
\end{itemize}
Then there exists $i\geq1$ such that $S\subset T_i$.
\end{lem}
\begin{proof}
As $S$ is irreducible, $\ideal(S,E)$ is a saturated prime ideal of $\an(E)$. The family of saturated ideals $\{\ideal(T_i,E)\}_{i\geq1}$ is locally finite and $\bigcap_{i\geq1}\ideal(T_i,E)=\ideal(T,E)\subset \ideal(S,E)$. By Lemma \ref{intpri} $\ideal(T_i,E)\subset\ideal(S,E)$ for some $i\geq1$, so $S\subset\ceros(\ideal(S,E))\subset\ceros(\ideal(T_i,E))=T_i$, as required.
\end{proof}

As a straightforward consequence of Lemma \ref{ni} and Theorem \ref{irredcomp2} proved below, we have:

\begin{cor}\label{nic}
Let $S\subset M$ be an amenable $C$-semianalytic set and let $\{S_i\}_{i\geq1}$ be the family of the irreducible components of $S$. Then $S_k\not\subset\bigcup_{i\neq k}S_i$ for each $k\geq1$.
\end{cor}

\subsubsection{Proof of Theorem \em\ref{irredcomp2}}
We divide the proof into two parts:

\noindent{\em Existence of the weak irreducible components.} Let $\ideal(S)=\bigcap_{i\geq1}\gtp_i$ be a locally finite (irredundant) primary decomposition of $\ideal(S)$. As $\ideal(S)$ is a radical ideal, the ideals $\gtp_i$ are by Corollary \ref{lemmaidreal} prime ideals of the ring $H^0(S,\an_M|_S)$. We have:
\begin{itemize}
\item $S_i:=\ceros(\gtp_i)\subset S$ is by Lemma \ref{sheaf} an $S$-tame $C$-semianalytic set.
\item $S=\bigcup_{i\geq 1}S_i$.
\item The family $\{S_i\}_{i\geq1}$ is locally finite in $S$.
\end{itemize}

We claim: \em Each $S$-tame $C$-semianalytic set $S_i$ is irreducible and $\ideal(\ceros(\gtp_i),S)=\gtp_i$ for $i=1,\ldots,\ell$\em. 

By Lemma \ref{clue3} it is enough to show that $\ideal(\ceros(\gtp_i),S)\subset\gtp_i$ for each $i\geq1$. As the primary decomposition is irredundant, $(\bigcap_{j\neq i}\gtp_j)\setminus\gtp_i\neq\varnothing$. Pick $g_i\in(\bigcap_{j\neq i}\gtp_j)\setminus\gtp_i$. Observe that $h_ig_i\in \ideal(S)\subset\gtp_i$ for each $h_i\in \ideal(\ceros(\gtp_i),S)$ because $h_ig_i$ vanishes identically on $S=S_i\cup\bigcup_{j\neq i}S_j$. As $g_i\not\in\gtp_i$, we conclude $h_i\in\gtp_i$, that is, $\ideal(\ceros(\gtp_i),S)\subset\gtp_i$.

The $S$-tame $C$-semianalytic sets $S_i$ for $i\geq1$ satisfy conditions (1), (3), (4) in Definition \ref{irredcompwtame}. Let us check that they also satisfy condition (2).

Indeed, let $S_i\subset T\subset S$ be an irreducible amenable $C$-semianalytic set. By Lemma \ref{ni} there exists $j\geq1$ such that $S_i\subset T\subset S_j$, so $\gtp_j\subset\gtp_i$. As $\gtp_i$ is a minimal prime ideal between those containing $\ideal(S)$, we deduce $\gtp_j=\gtp_i$, so $S_i=T=S_j$, as required.

\noindent{\em Uniqueness of weak irreducible components.} Let $\{S_i\}_{i\geq1}$ be the family of weak irreducible components constructed above and let $\{T_j\}_{j\geq1}$ be another family of weak irreducible components of $S$ satisfying the conditions in Definition \ref{irredcompwtame}. By Lemma \ref{ni} each $T_i\subset S_j$ for some $j\geq1$. By condition (2) in Definition \ref{irredcompwtame} we have $T_i=S_j$. It follows straightforwardly by Lemma \ref{ni} that $\{S_i\}_{i\geq1}=\{T_j\}_{j\geq1}$, as required.
\qed

\subsubsection{Neat behavior of the weak irreducible components}
We show that the behavior of the Zariski closure of an amenable $C$-semianalytic set $S$ in a small enough open neighborhood $U\subset M$ of $S$ with respect to the weak irreducible components is neat.

\begin{prop}\label{neatwirred}
Let $S\subset M$ be an amenable $C$-semi\-analytic set. There exist an open neighborhood $U\subset M$ of $S$ such that if $X:=\ol{S}^{\zar}$ and $\{X_i\}_{i\geq1}$ are the irreducible components of $X$, then $\{S_i:=X_i\cap S\}_{i\geq1}$ is the family of the weak irreducible components of $S$ and $X_i:=\ol{S_i}^{\zar}$ for $i\geq1$. In particular, if $S$ is a global $C$-semianalytic subset of $M$, each $S_i$ is a global $C$-semianalytic subset of $U$.
\end{prop}
\begin{proof}
For each $i\geq1$ there exist by Lemma \ref{sheaf} an open neighborhood $U_i\subset M$ of $S$ and a $C$-analytic set $Y_i$ such that $S_i=Y_i\cap S$. 

As the family $\{S_i\}_{i\geq1}$ is locally finite in $S$, we may assume after substituting $M$ by an open neighborhood of $S$ that \em the family $\{S_i\}_{i\geq1}$ is locally finite in $M$\em. Indeed, for each $x\in S$ let $U^x\subset X$ be an open neighborhood of $x$ such that only finitely many $S_i$ meet $U^x$. It is enough to take $M':=\bigcup_{x\in Y}U^x$ instead of $M$.

By Lemma \ref{neighs} there exists a locally finite family $\{U_i'\}_{i\geq1}$ of open neighborhoods $U_i'\subset U_i$ of $S_i$. As $S_i$ is an amenable $C$-analytic subset of $U_i'$, the Zariski closure $Y_i'\subset Y_i\cap U_i'$ of $S_i$ in $U_i'$ has its same dimension. Each $Y_i'$ is a closed subset of $U_i'$ and it holds 
$$
S_i=Y_i'\cap S\subset\cl_M(Y_i')\cap S\subset\cl_M(Y_i)\cap S\cap U_i=\cl_{U_i}(Y_i)\cap S=Y_i\cap S=S_i,
$$
that is, $\cl_M(Y_i')\cap S=Y_i\cap S$. By Lemma \ref{bigneigh} there exists an open neighborhood $U\subset M$ of $S$ such that $Y_i'':=Y_i'\cap U$ is a closed subset of $U$ for $i\geq1$. Let $h_i$ be an analytic equation of $Y_i'$ in $\an(U_i')$. Notice that $Y_i''$ is a $C$-analytic subset of $U$ because it is the zero set of the coherent sheaf on $U$:
$$
{\mathcal F}_x:=\begin{cases}
h_i\an_{U,x}&\text{ if $x\in Y_i''$},\\
\an_{U,x}&\text{ otherwise}.
\end{cases}
$$
As $S_i$ is irreducible, the Zariski closure $X_i\subset Y_i''$ of $S_i$ in $U$ is an irreducible $C$-analytic set. Notice that 
\begin{itemize}
\item[(1)] $S_i\subset S\cap X_i\subset S\cap Y_i=S_i$, so $X_i\cap S=S_i$,
\item[(2)] The family $\{X_i\}_{i\geq1}$ is locally finite because $X_i\subset U_i'$ and the family $\{U_i'\}_{i\geq1}$ is locally finite. 
\end{itemize}
Putting all together, we are done.
\end{proof}

\subsection{Irreducible components of an amenable $C$-semianalytic set}
Our next purpose is to prove the existence and uniqueness of the family of the irreducible components of an amenable $C$-semianalytic set $S\subset M$. The strategy is to show that the weak irreducible components are in fact the irreducible components of $S$. Once this is done Theorem \ref{bijection} follows from Theorem \ref{irredcomp2} and Proposition \ref{neatirred} is only a reformulation of Proposition \ref{neatwirred}. In the following $S\subset M$ denotes an amenable $C$-semianalytic set.

\begin{thm}\label{irredcomp3}
The family of the weak irreducible components of $S$ is a family of irreducible components of $S$. In addition, the family of the irreducible components of $S$ is unique.
\end{thm}

The crucial step to prove Theorem \ref{irredcomp3} is the following result.

\begin{thm}\label{tamend}
The weak irreducible components of $S$ are amenable $C$-semianalytic sets and constitute a locally finite family of $M$.
\end{thm}

Before proving Theorem \ref{tamend} we need some preliminary work. We denote the family of the weak irreducible components of $S$ with $\{S_i\}_{i\geq1}$. 

\begin{lem}\label{neigh}
For each $i\geq1$ there exists an open $C$-semianalytic set $U$ such that $S\cap U=S_i\cap U$ is a real analytic manifold of dimension $\dim(S_i)$.
\end{lem}
\begin{proof}
For simplicity we prove the result for $S_1$. By Corollary \ref{nic} $S_1\not\subset\bigcup_{j>1}S_j$. Let $V\subset M$ be an open neighborhood of $S$ such that $S_1$ is an amenable $C$-semianalytic subset of $V$ and let $X_1$ be the Zariski closure of $S_1$ in $V$. By Lemma \ref{sheaf} we may assume after shrinking $V$ that there exists $h\in\bigcap_{j>1}\ideal(S_1,S)\cap\an(V)$ such that $\ceros(h)\cap S=\bigcup_{j>1}\ceros(\ideal(S_1,S))$. As $S_1$ is irreducible and $S_1\not\subset\bigcup_{j>1}S_j$, we deduce by Lemma \ref{charirred} that $\dim(\ceros(h)\cap S_1)<\dim(S_1)$. Pick a point $x\in\Int_{\Reg(X_1)}(S_1\cap\Reg(X_1)\setminus\bigcup_{j>1}S_j)$ and let $U$ be an open $C$-semianalytic neighborhood of $x$ in $M$ such that 
$$
S\cap U=\Int_{\Reg(X_1)}\Big(S_1\cap\Reg(X_1)\setminus\bigcup_{j>1}S_j\Big)\cap U.
$$
Observe that $U$ satisfies the conditions in the statement.
\end{proof}

\begin{prop}\label{zclc}
The equality $\dim(\ol{S_i}^{\zar})=\dim(S_i)$ holds for $i\geq1$ and the family $\{\ol{S_i}^{\zar}\}_{i\geq1}$ is locally finite after eliminating repetitions.
\end{prop}
\begin{proof}
As $S$ is amenable, there exists by Theorem \ref{algor1} and \ref{clear} (vi) and index $r\geq1$ such that $S=\bigcup_{k=1}^rT_k(S)$ and $\dim(T_{k+1}(S))<\dim(T_k(S))$. Recall that each $T_k(S)$ is a real analytic manifold. For each $i\geq1$ there exists by Lemma \ref{neigh} an open $C$-semianalytic set $U_i$ such that $S\cap U_i=S_i\cap U_i$ is a real analytic manifold of dimension $\dim(S_i)$. As $S\cap U_i=\bigcup_{k=1}^rT_k(S)\cap U_i$, there exists $1\leq k\leq r$ such that $\dim(T_k(S))=\dim(T_k(S)\cap U_i)=\dim(S\cap U_i)=\dim(S_i)$. By Lemma \ref{charirred}(iii) $\ol{S_i}^{\zar}=\ol{S\cap U_i}^{\zar}$ is an irreducible component of $\ol{T_k(S)}^{\zar}$. As $S\cap U_i$ is an amenable $C$-semianalytic set, $\dim(\ol{S_i}^{\zar})=\dim(\ol{S\cap U_i}^{\zar})=\dim(S\cap U_i)=\dim(S_i)$.

As the family $\{\ol{T_k(S)}^{\zar}\}_{k=1}^r$ is finite and the irreducible components of each $\ol{T_k(S)}^{\zar}$ constitute a locally finite family, we conclude that the family $\{\ol{S_i}^{\zar}\}_{i\geq1}$ is locally finite after eliminating repetitions.
\end{proof}

\begin{cor}
We have:
\begin{itemize}
\item[(i)] Let $U\subset M$ be an open neighborhood of $S$ and let $\{X_j\}_{j\geq1}$ be the irreducible components of the $C$-analytic set $X:=\ol{S}^{\zar}_U$. Then for each $j\geq1$ there exists $i\geq1$ such that $X_j=\ol{S_i}^{\zar}_U$ and $\I(X_j,X)=\I(S_i,S)\cap\an(X)$.
\item[(ii)] There exist an open neighborhood $V\subset M$ of $S$ such that if $X_i:=\ol{S_i}^{\zar}_V$, then $\{X_i\}_{i\geq1}$ is the family of the irreducible components of $X:=\ol{S}^{\zar}_V$ and $\I(X_i,X)=\I(S_i,S)\cap\an(X)$ for $i\geq1$.
\end{itemize}
\end{cor}
\begin{proof}
(i) For simplicity consider $U=M$. By Proposition \ref{zclc} $\dim(\ol{S_i}^{\zar})=\dim(S_i)$ holds and the family $\{\ol{S_i}^{\zar}\}_{i\geq1}$ is locally finite after eliminating repetitions. Thus, $Z:=\bigcup_{i\geq1}\ol{S_i}^{\zar}\subset X$ is a $C$-analytic set that contains $S$, so $Z=X$. By Lemma \ref{ni} there exists $i\geq1$ such that $X_j\subset\ol{S_i}^{\zar}\subset X$, so $X_j=\ol{S_i}^{\zar}$. Consequently
$$
\I(S_i,S)\cap\an(X)=\{f\in\an(X):\ f|_{S_i}=0\}=\{f\in\an(X):\ f|_{\ol{S_i}^{\zar}}=0\}=\I(X_j,X).
$$ 

(ii) Apply (i) to the open neighborhood $V=U$ of $S$ constructed in Proposition \ref{neatwirred}.
\end{proof}

We are ready to prove Theorem \ref{tamend}.

\begin{proof}[Proof of Theorem \em \ref{tamend}]
We divide the proof in two parts:

\noindent{\bf Part 1.} {\em Amenability in $M$ of the weak irreducible components of $S$.} 
Let $S_1$ be a weak irreducible component of $S$ and let us show that $S_1$ is an amenable $C$-semianalytic set. The proof of this part is conducted in several steps:

\paragraph{}\label{reds1} 
Let $X$ be the Zariski closure of $S_1$. We prove next that we may assume from the beginning: \em $X=\ol{S}^{\zar}$. In particular, the Zariski closure of $S$ is irreducible\em. To that end, we show: \em $S_1$ is a weak irreducible component of the amenable $C$-semi\-analytic set $S':=S\cap X$\em. Once this is done we substitute $S$ by $S'$.

Let $\{S'_k\}_{k\geq1}$ be the family of weak irreducible components of $S'$. By Lemma \ref{ni}, there exists $k\geq1$ such that $S_1\subset S'_k$. We claim: \em $S'_k$ is an $S$-tame $C$-semianalytic set\em. 

It is enough to check that $S'_k=\ceros(h)$ for some analytic function $h$ on an open neighborhood of $S$. Let $h_0$ be an analytic function on an open neighborhood $W$ of $S'$ such that $S'_k=S'\cap\ceros(h_0)$. As $X$ is a $C$-analytic subset of $M$, it holds that $X\cap W$ is a $C$-analytic subset of $M\setminus(X\setminus W)$. By Cartan's Theorem B $h_0^2|_{X\cap W}$ is the restriction to $X\cap W$ of an analytic function $h$ on $M\setminus(X\setminus W)$ such that $\ceros(h)=\ceros(h_0)\cap X\cap W$. Observe that $S\subset M\setminus(X\setminus W)$. Consequently, $S'_k=S\cap\ceros(h)$ is a $S$-tame $C$-semianalytic set. 

By Lemma \ref{ni} there exists a weak irreducible component $S_j$ of $S$ such that $S_1\subset S'_k\subset S_j$ and we conclude $S'_k=S_1$.

\paragraph{}Let $\widetilde{X}$ be an irreducible complexification of the irreducible $C$-analytic set $X$. Recall that if $\Omega$ is an open subset of $\widetilde{X}$, then the irreducible components of $\Omega$ are all pure dimensional and coincide with the closures of the connected components of the complex analytic manifold $\Omega\setminus\Sing(\widetilde{X})$. Let $(\widetilde{X}^\R,\an_{\widetilde{X}}^\R)$ be the underlying real analytic structure of $(\widetilde{X},\an_{\widetilde{X}})$. Notice that $\Reg(\widetilde{X}^\R)=\Reg(\widetilde{X})$ and the irreducible components of $\Omega^\R$ arise as the underlying real structures of the irreducible components of $\Omega$.

\paragraph{}Let $(Z,\an_Z)$ be a Stein complexification of $(\widetilde{X}^\R,\an_{\widetilde{X}}^\R)$ and let $\sigma:(Z,\an_Z)\to(Z,\an_Z)$ be an anti-involution whose fixed locus is $\widetilde{X}^\R$. Recall that $\Sing(\widetilde{X}^\R)=\Sing(Z)\cap \widetilde{X}^\R$. Denote the reduction of $(Z,\an_Z)$ with $(Z_1,\an_{Z_1}):=(Z,\an_Z^r)$. Observe that $\sigma$ induces an anti-involution on $(Z_1,\an_{Z_1})$ whose fixed part space $(Z^\sigma_1,\an_{Z^\sigma_1})$ satisfies $Z^\sigma_1=\widetilde{X}^\R$ and $\an_{Z^\sigma_1}$ is a quotient (coherent) sheaf of $\an_{\widetilde{X}}^\R$. For each $z\in\widetilde{X}^\R$ it holds $\an_{Z^\sigma_1,z}\cong\an_{\widetilde{X},z}^\R/\gtn(\an_{\widetilde{X},z}^\R)$ where $\gtn(\an_{\widetilde{X},z}^\R)$ is the ideal of nilpotents elements of $\an_{\widetilde{X},z}^\R$. By \cite[V.\S3]{gare} $(Z_1,\an_{Z_1})$ is a Stein space.

Let $\pi:Y\to Z_1$ be the normalization of $(Z_1,\an_{Z_1})$. As $(Z_1,\an_{Z_1})$ is Stein, $(Y,\an_Y)$ is by \cite{n3} Stein. The anti-involution on $Z_1$ extends to an anti-involution $\widehat{\sigma}$ on $Y$ such that $\pi\circ\widehat{\sigma}=\sigma\circ\pi$, see \cite[IV.3.10]{gmt}. Denote the set of fixed points of $\widehat{\sigma}$ with $Y^{\widehat{\sigma}}:=\{y\in Y:\ \widehat{\sigma}(y)=y\}$. As the restriction $\pi|:Y\setminus\pi^{-1}(\Sing(Z_1))\to Z_1\setminus\Sing(Z_1)$ is an invariant holomorphic diffeomorphism, 
$$
\pi(Y^{\widehat{\sigma}}\setminus\pi^{-1}(\Sing(Z_1))=Z_1^{\sigma}\setminus\Sing(Z_1)=\Reg(\widetilde{X}^\R).
$$
As $\pi$ is proper, $\pi(Y^{\widehat{\sigma}})=\widetilde{X}^\R$. Note that 
$\pi^{-1}(S)\cap Y^{\widehat{\sigma}}$ is an amenable $C$-semianalytic subset of $Y^{\widehat{\sigma}}$. 

\paragraph{}By Proposition \ref{neatwirred} there exists an open neighborhood $V\subset M$ of $S$ such that if $T:=\ol{S}^{\zar}_V$ and $\{T_i\}_{i\geq1}$ are the irreducible components of $T$, we may assume $S_i=T_i\cap S$ for $i\geq1$. Observe that $\dim(T_1)=\dim(S_1)=\dim(X)$, so $T_1$ is an irreducible component of $X\cap V$.

\paragraph{} Let $\Omega$ be an open neighborhood of $X\cap V$ in $\widetilde{X}$ such that $\Omega\cap X=V\cap X$ and for each irreducible component $X'$ of $X\cap V$ there exists an irreducible component $\Omega'$ of $\Omega$ such that $\Omega'\cap X=X'$ (see \cite[\S8. Prop. 11]{wb}). As $S\subset\Omega$, it holds $\pi^{-1}(S)\subset\pi^{-1}(\Omega)$. Let $\Omega_1$ be the irreducible component of $\Omega$ such that $\Omega_1\cap X=T_1$. Recall that $\Omega_1^\R$ is an irreducible component of $\Omega^\R$. Let $\Theta$ be an open neighborhood of $\Omega$ in $Z_1$ such that $\Theta\cap\widetilde{X}^\R=\Omega^\R$ and for each irreducible component ${\Omega'}^\R$ of $\Omega^\R$ there exists an irreducible component $\Theta'$ of $\Theta$ such that $\Theta'\cap\widetilde{X}^\R={\Omega'}^\R$. Let $\Theta_1$ be the irreducible component of $\Theta$ such that $\Theta_1\cap\widetilde{X}^\R=\Omega_1^\R$. Let $Y_1'$ be the connected component of $Y':=\pi^{-1}(\Theta)$ such that $\pi(Y_1')=\Theta_1$ (see Remark \ref{norm1}). 

\paragraph{} As $S\subset X\cap V\subset\Omega\subset\Theta$, we have $\pi^{-1}(S)\subset\pi^{-1}(\Theta)=Y'$. As $Y_1'$ is a connected component of $Y'$, the intersection $R_1:=Y_1'\cap\pi^{-1}(S)\cap Y^{\widehat{\sigma}}$ is an open and closed subset of $R:=\pi^{-1}(S)\cap Y^{\widehat{\sigma}}$, so $R_1$ is a union of connected components of $R$. We claim: $\pi(R_1)=S_1$.

Indeed, $\Omega_1^\R$ is the closure of the connected component $\Omega_1^\R\setminus\Sing(Z_1)$ of $\Omega^\R\setminus\Sing(Z_1)$. As the restriction $\pi|:Y\setminus\pi^{-1}(\Sing(Z_1))\to Z_1\setminus\Sing(Z_1)$ is an invariant holomorphic diffeomorphism and $\Omega_1^\R=\Theta_1\cap\widetilde{X}^\R$, we conclude 
$$
\pi((Y_1'\setminus\pi^{-1}(\Sing(Z_1)))\cap Y^{\widehat{\sigma}})=(\Theta_1\setminus\Sing(Z_1))\cap Z_1^\sigma=\Omega_1^\R\setminus\Sing(Z_1).
$$
As $\pi$ is proper, $\pi(Y_1'\cap Y^{\widehat{\sigma}})=\Omega_1^\R$. Thus,
$$
\pi(R_1)=\pi(Y_1'\cap Y^{\widehat{\sigma}}\cap\pi^{-1}(S))=\pi(Y_1'\cap Y^{\widehat{\sigma}})\cap S
=\Omega_1^\R\cap S=\Omega_1^\R\cap X\cap S=T_1\cap S=S_1. 
$$

\paragraph{} As $R_1$ is a union of connected components of $R$, we deduce by Theorem \ref{properint-tame}(ii) that $S_1=\pi(R_1)$ is an amenable $C$-semianalytic subset of $Z_1^\sigma$. By \cite[II.4.10]{gmt} $(X,\an_X)$ is a closed subspace of $(\widetilde{X}^\R,\an_{\widetilde{X}}^\R)$. As $\an_{X,x}$ contains no nilpotent element for each $x\in X$ (recall that we have considered on $X$ the well-reduced structure, see \ref{wrs}), it holds that $(X,\an_X)$ is a closed subspace of $(Z_1^\sigma,\an_{Z_1^\sigma})$. Consequently, $S_1$ is an amenable $C$-analytic subset of $X$. As $X\subset M$ is a $C$-analytic set, $S_1$ is by Cartan's Theorem B an amenable $C$-analytic subset of $M$, as required.

\noindent{\bf Part 2.} {\em Local finiteness in $M$ of the family of the weak irreducible components of $S$.} By Proposition \ref{zclc} it is enough to prove the following: \em Let $X:=\ol{S_i}^{\zar}$ for some $i\geq1$ and let ${\mathfrak F}:=\{j\geq1:\ \ol{S_j}^{\zar}=X\}$. Then $\{S_j\}_{j\in{\mathfrak F}}$ is locally finite in $M$\em.

We may assume that the Zariski closure of $S$ is $X$ (see \ref{reds1}). Let $(\widetilde{X},\an_{\widetilde{X}})$ be a complexification of $(X,\an_X)$ an let $(Y,\pi)$ be its normalization. Let $V\subset M$ be an open neighborhood of $S$ such that the irreducible components $\{T_i\}_{i\geq1}$ of $T:=\ol{S}^{\zar}_V$ satisfy $S_i=S\cap T_i$ (use Proposition \ref{neatwirred}). The family $\{T_j\}_{j\in{\mathfrak F}}$ is a collection of irreducible components of $X\cap V$ of its same dimension. Let $\Omega$ be an open neighborhood of $X\cap V$ in $\widetilde{X}$ such that $\Omega\cap X=X\cap V$ and for each irreducible component $X'$ of $X\cap V$ there exists an irreducible component $\Omega'$ of $\Omega$ such that $\Omega'\cap X=X'$. Let $\Omega_j$ be the irreducible component of $\Omega$ such that $\Omega_j\cap V=T_j$ for $j\in{\mathfrak F}$. 

By Remark \ref{norm1} $(\Theta:=\pi^{-1}(\Omega),\pi|_{\Theta})$ is the normalization of $\Omega$ and for each $j\in{\mathfrak F}$ there exists a connected component $\Theta_j$ of $\Theta$ such that $\pi(\Theta_j)=\Omega_j$ and $(\Theta_j,\pi|_{\Theta_j})$ is the normalization of $\Omega_j$. As $\pi^{-1}(S)\subset\Theta$, the intersection $\pi^{-1}(S)\cap\Theta_j$ is a union of connected components of $\pi^{-1}(S)$. We claim: \em $\pi^{-1}(S)\cap\Theta_j=\pi^{-1}(S_j)\cap\Theta_j$. In particular, each connected component of $\pi^{-1}(S_j)\cap\Theta_j$ is a connected component of $\pi^{-1}(S)$.\em

As $\pi(\Theta_j)=\Omega_j$ and $\Omega_j\cap S=\Omega_j\cap X\cap S=T_j\cap S=S_j$, we deduce 
$$
\pi^{-1}(S_j)\cap\Theta_j=\pi^{-1}(\Omega_j)\cap\pi^{-1}(S)\cap\Theta_j=\pi^{-1}(S)\cap\Theta_j.
$$

Fix $j\in{\mathfrak F}$. As $X$ is the Zariski closure of $S_j$ in $M$, there exists by Theorem \ref{dpm}(i) a connected component $R_j$ of $\pi^{-1}(S_j)$ such that $\pi(R_j)=S_j$. As $S_j\subset\Omega_j$, we deduce that $R_j\subset\Theta_j$, so $R_j$ is a connected component of $\pi^{-1}(S_j)\cap\Theta_j$. Consequently, $R_j$ is a connected component of $\pi^{-1}(S)$. As $\pi^{-1}(S)$ is a semianalytic subset of the underlying real analytic structure $(Y^\R,\an_Y^\R)$ of $(Y,\an_Y)$, the family of its connected components is by \cite[2.7]{bm} locally finite. Consequently, the family $\{R_j\}_{j\in{\mathfrak F}}$ is locally finite. By Lemma \ref{proylc} the family $\{S_j=\pi(R_j)\}_{j\in{\mathfrak F}}$ is locally finite in $X$, so it is locally finite in $M$, as required.
\end{proof}

\begin{proof}[Proof of Theorem \em \ref{irredcomp3}]
By Definition \ref{irredcompwtame} and Theorem \ref{tamend} the family of weak irreducible components of $S$ is a family of irreducible components of $S$. Let us prove next that the family of irreducible components of $S$ is unique.

Let $\{S_i\}_{i\geq1}$ be the family of the weak irreducible components of $S$ and let $\{T_j\}_{j\geq1}$ be a family of irreducible components of $S$ satisfying the conditions of Definition \ref{irredcomptame}. By Lemma \ref{ni} there exists an index $i\geq1$ such that $T_j\subset S_i$ for each $j\geq1$. By condition (2) in Definition \ref{irredcomptame} we have $T_j=S_i$. By Corollary \ref{nic} we conclude $\{T_j\}_{j\geq1}=\{S_i\}_{i\geq1}$, so the family of irreducible components of $S$ is unique.
\end{proof}

\begin{example}\label{nonpure}
The irreducible components of a pure dimensional amenable $C$-semianalytic set need not to be pure dimensional. Let $S:=T_1\cup T_2\cup T_3\subset\R^3$ where
\begin{multline*}
T_1:=[-1,1]\times[-2,2]\times\{0\},\quad T_2:=[-2,-1]\times\{-1,1\}\times[-1,1],\\
\&\quad T_3:=[1,2]\times\{-1,1\}\times[-1,1]
\end{multline*}
By Proposition \ref{neatwirred} and Theorem \ref{irredcomp3} the irreducible components of $X$ are the intersections $S_1:=S\cap\{x_3=0\}$, $S_2:=S\cap\{x_2=1\}$ and $S_3:=S\cap\{x_2=-1\}$ and none of them is pure dimensional.
\end{example}

\begin{figure}[ht]
\centering
\begin{tikzpicture}[x=.275cm,y=.275cm]

\draw (10,6) -- (10,14) -- (13,16) -- (13,8) -- (10,6);
\draw[fill=black!40!white] (10,6) -- (10,14) -- (13,16) -- (13,8) -- (10,6);

\draw (18,6) -- (18,14) -- (21,16) -- (21,8) -- (18,6);
\draw[fill=black!40!white] (18,6) -- (18,14) -- (21,16) -- (21,8) -- (18,6);

\draw (0,6) -- (6,10) -- (22,10) -- (16,6) -- (0,6);
\draw[fill=black!20!white] (0,6) -- (6,10) -- (22,10) -- (16,6) -- (0,6);

\draw (1,0) -- (1,8) -- (4,10) -- (4,2) -- (1,0);
\draw[fill=black!40!white] (1,0) -- (1,8) -- (4,10) -- (4,2) -- (1,0);

\draw (9,0) -- (9,8) -- (12,10) -- (12,2) -- (9,0);
\draw[fill=black!40!white] (9,0) -- (9,8) -- (12,10) -- (12,2) -- (9,0);

\draw[dashed] (1,4) -- (13,12);
\draw[dashed] (9,4) -- (21,12);


\draw[dashed] (29,4) -- (41,12);
\draw[dashed] (37,4) -- (49,12);
\draw[thick] (29,4) -- (32,6);
\draw[thick] (38,10) -- (41,12);
\draw[thick] (37,4) -- (40,6);
\draw[thick] (46,10) -- (49,12);

\draw (33,8) -- (33,16) -- (36,18) -- (36,10) -- (33,8);
\draw[fill=black!40!white] (33,8) -- (33,16) -- (36,18) -- (36,10) -- (33,8);
\draw[thick] (27,8) -- (33,12);

\draw (51,8) -- (51,16) -- (54,18) -- (54,10) -- (51,8);
\draw[fill=black!40!white] (51,8) -- (51,16) -- (54,18) -- (54,10) -- (51,8);

\draw (28,6) -- (34,10) -- (50,10) -- (44,6) -- (28,6);
\draw[fill=black!20!white] (28,6) -- (34,10) -- (50,10) -- (44,6) -- (28,6);

\draw[thick] (29,4) -- (32,6);
\draw[thick] (38,10) -- (41,12);
\draw[dashed] (29,4) -- (41,12);
\draw[dashed] (37,4) -- (49,12);

\draw (24,2) -- (24,10) -- (27,12) -- (27,4) -- (24,2);
\draw[fill=black!40!white] (24,2) -- (24,10) -- (27,12) -- (27,4) -- (24,2);

\draw (42,2) -- (42,10) -- (45,12) -- (45,4) -- (42,2);
\draw[fill=black!40!white] (42,2) -- (42,10) -- (45,12) -- (45,4) -- (42,2);
\draw[thick] (45,8) -- (51,12);


\draw (2,14) node{$S$};
\draw (17,8) node{$T_1$};
\draw (16,12) node{$T_2$};
\draw (6,4) node{$T_3$};

\draw (39,8) node{$S_1$};
\draw (48,6) node{$S_2$};
\draw (30,12) node{$S_3$};

\draw[dashed,->,thick] (30.5,9.5) arc (180:270:0.5cm) -- (34,8);
\draw[dashed,->,thick] (47,9.5) arc (0:-90:0.5cm) -- (44,8);

\end{tikzpicture}
\caption{Irreducible components $S_1$, $S_2$ and $S_3$ of $S$ (Example \ref{nonpure})}
\end{figure}
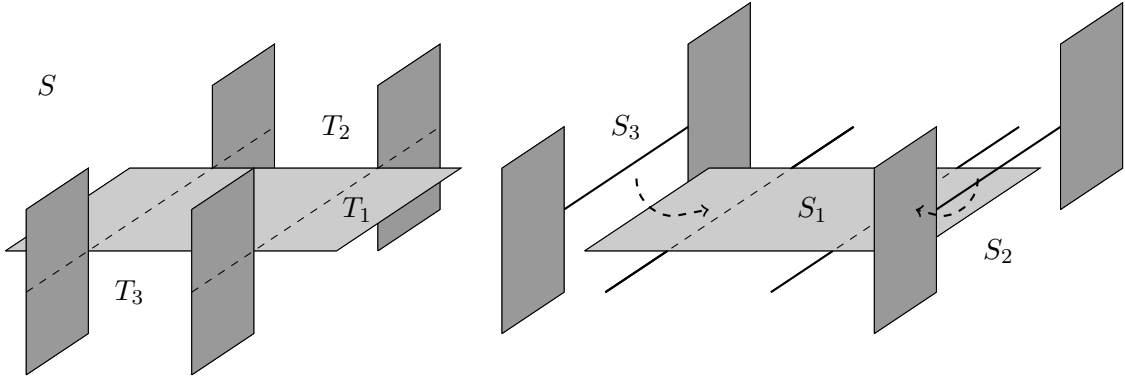

\subsection{Irreducible components vs connected components}
Let $S\subset M$ be an amenable $C$-semianalytic set and let $\{S_i\}_{i\geq1}$ be the family of its irreducible components. Let $X$ be the Zariski closure of $S$ and let $(\widetilde{X},\sigma)$ be a complexification of $X$ together with the anti-involution $\sigma:\widetilde{X}\to\widetilde{X}$ whose set of fixed points is $X$. Let $(Y,\pi)$ be the normalization of $\widetilde{X}$ and let $\widehat{\sigma}:Y\to Y$ be the anti-holomorphic involution induced by $\sigma$ in $Y$, which satisifies $\pi\circ\widehat{\sigma}=\sigma\circ\pi$. We study next \em if the irreducible components of $S$ can be computed as the images of some of the connected components of $\pi^{-1}(S)$\em. The following example shows that this is not true in general. However, we show in Proposition \ref{dpm2} that under certain conditions (achieved in Proposition \ref{neatwirred}), the result is true.

\begin{example}
Consider the amenable $C$-semianalytic subset $S:=S_0\cup S_1\cup S_2\cup S_3$ of $\R^4$ where 
\begin{align*}
S_0&:=\{x^2-zy^2=0,z>0\},\\
S_1&:=\{x=0,y=0,w=0,z<0\},\\
S_2&:=\{x=0,y=0,w+z=-1,z<0\},\\
S_3&:=\{x=0,y=0,w-z=1,z<0\}.
\end{align*}
The irreducible components of $S$ are $S_0,S_1,S_2$ and $S_3$. The Zariski closure of $S$ is $X:=\{x^2-zy^2=0\}\subset\R^4$. Consider the complexification $\widetilde{X}:=\{x^2-zy^2=0\}\subset\C^4$ and its normalization $\pi:Y:=\C^3\to\widetilde{X},\ (s,t,w)\mapsto(st,s,t^2,w)$. Observe that $\pi^{-1}(S)=\bigcup_{i=0}^6T_i$ where 
$$
\begin{array}{lll}
T_0:=\{(s,t,w)\in\R^3:\ t\neq0\},\\[4pt]
T_1:=\{(0,\sqrt{-1}t,0)\in\C^3:\ t>0\},&T_4:=\{(0,-\sqrt{-1}t,0)\in\C^3:\ t>0\},\\[4pt]
T_2:=\{(0,\sqrt{-1}t,-1+t)\in\C^3:\ t>0\},&T_5:=\{(0,-\sqrt{-1}t,-1+t)\in\C^3:\ t>0\},\\[4pt]
T_3:=\{(0,\sqrt{-1}t,1-t)\in\C^3:\ t>0\}, &T_6:=\{(0,-\sqrt{-1}t,1-t)\in\C^3:\ t>0\}.
\end{array}
$$
In addition $\pi(T_0)=S_0$, $\pi(T_1)=\pi(T_4)=S_1$, $\pi(T_2)=\pi(T_5)=S_2$ and $\pi(T_3)=\pi(T_6)=S_3$. Observe that $T_1\cap T_2\cap T_3=\{(0,\sqrt{-1},0)\}$ and $T_4\cap T_5\cap T_6=\{(0,-\sqrt{-1},0)\}\}$. Consequently, $\pi^{-1}(S)$ has three connected components $C_0:=T_0$, $C_1:=T_1\cup T_2\cup T_3$ and $C_2:=T_4\cup T_5\cup T_6$, while $S$ has four irreducible components.
\end{example}

Before stating the following result we refer the reader to Proposition \ref{neatwirred}.

\begin{prop}\label{dpm2}
Suppose that if $\{X_i\}_{i\geq1}$ are the irreducible components of $X$, we have $S_i=X_i\cap S$ and $\ol{S_i}^{\zar}=X_i$ for $i\geq1$. Then for each $i\geq1$ there exists a connected components $T_i$ of $\pi^{-1}(S)$ such that $\pi(T_i)=S_i$.
\end{prop}
\begin{proof}
Let $\widetilde{X}'\subset\widetilde{X}$ be an open neighborhood of $X$ such that each irreducible component $X_i$ of $X$ is the intersection with $X$ of an irreducible component $\widetilde{X}_i'$ of $\widetilde{X}'$ (see \cite[\S8. Prop. 11]{wb}). 

By Remark \ref{norm1} $(Y':=\pi^{-1}(\widetilde{X}'),\pi|_{Y'})$ is the normalization of $\widetilde{X}'$ and for each $i\geq1$ there exists a connected component $Y'_i$ of $Y'$ such that $\pi(Y'_i)=\widetilde{X}'_i$ and $(Y'_i,\pi|_{Y'_i})$ is the normalization of $\widetilde{X}'_i$. As $S\subset X$, we have $\pi^{-1}(S)\subset Y'$. Consequently, the intersection $\pi^{-1}(S)\cap Y'_i$ is a union of connected components of $\pi^{-1}(S)$.

As $S_i$ is an amenable $C$-semianalytic set, there exists by Theorem \ref{dpm} a connected component $T_i$ of $\pi^{-1}(S)\cap Y'_i$ such that $\pi(T_i)=S_i$. Observe that $T_i$ is a connected component of $\pi^{-1}(S)$, as required. 
\end{proof}

\appendix
\section{Some properties of locally finite families}\label{app}

We recall in this appendix certain properties of locally finite families of a topological space for the sake of completeness.

\begin{lem}\label{neighs}
Let $X$ be a paracompact second countable topological space and let $\{T_k\}_{k\geq1}$ be a locally finite family of subsets of $X$. For each $k\geq1$ let $V_k$ be an open neighborhood of $T_k$ in $X$. Then there exist open neighborhoods $U_k\subset V_k$ of $T_k$ in $X$ such that the family $\{U_k\}_{k\geq1}$ is locally finite in $X$.
\end{lem}
\begin{proof}
For each $x\in X$ there exists an open neighborhood $W^x$ of $x$ that meets only finitely many $T_k$. The family $\{W^x\}_{x\in X}$ is an open covering of $X$. Thus, it has an open refinement $\{W_\ell\}_{\ell\geq1}$ which is countable and locally finite in $X$. Define $U_k':=\bigcup_{W_\ell\cap T_k\neq\varnothing}W_\ell$ and observe that $T_k\subset U_k'$. We claim: \em The family $\{U_k'\}_{k\geq1}$ is locally finite in $X$\em.

Fix a point $x\in X$ and let $V^x$ be a neighborhood of $x$ which intersects finitely many $W_\ell$, say $W_{\ell_1},\ldots,W_{\ell_r}$. The union $\bigcup_{j=1}^rW_{\ell_j}$ meets only finitely many $T_k$, say $T_{k_1},\ldots,T_{k_s}$. If $k\neq k_1,\ldots,k_s$, the intersection $U_k'\cap V^x=\varnothing$. 

To finish take $U_k:=U_k'\cap V_k$.
\end{proof}

\begin{lem}\label{bigneigh}
Let $X$ be a topological space. Let $\{\Omega_k\}_{k\geq1}$ be a locally finite family of open subsets of $X$ and for each $k\geq1$ let $T_k\subset\Omega_k$ be a closed subset of $\Omega_k$. Let $Y$ be a subset of $X$ and suppose that $T_k\cap Y$ is closed in $Y$ for all $k\geq1$. Then there exists an open neighborhood $\Omega\subset X$ of $Y$ such that $\Omega\cap T_k$ is a closed subset of $\Omega$ for all $k\geq1$.
\end{lem}
\begin{proof}
As the family $\{\Omega_k\}_{k\geq1}$ is locally finite in $X$, so are the families $\{T_k\}_{k\geq1}$, $\{\cl(T_k)\}_{k\geq1}$ and $\{\cl(T_k)\setminus\Omega_k\}_{k\geq1}$. Thus,
$$
E:=\bigcup_{k\geq1}\cl(T_k)\qquad\text{and}\qquad C_k:=\bigcup_{j\neq k}\cl(T_j)\setminus\Omega_j
$$
are closed subset of $X$. Consider the open subset of $X$
$$
\Omega:=(X\setminus E)\cup\bigcup_{k\geq1}\Omega_k\setminus C_k.
$$
We check first: $Y\subset\Omega$. 

As $T_k\cap Y$ is closed in $Y$, we have $\cl(T_k)\cap Y=T_k\cap Y$, so
$$
C_k\cap Y=\bigcup_{j\neq k}(\cl(T_j)\cap Y)\setminus\Omega_j=\bigcup_{j\neq k}(T_j\cap Y)\setminus\Omega_j=\varnothing,
$$
for all $k\geq1$ and $E\cap Y=\bigcup_{k\geq1}T_k\cap Y$. As $T_k\subset\Omega_k$,
$$
Y\cap\Omega=\Big(Y\setminus\bigcup_{k\geq1}T_k\cap Y\Big)\cup\bigcup_{k\geq1}\Omega_k\cap Y=Y.
$$
Next, we show: \em Each intersection $T_\ell\cap\Omega$ is closed in $\Omega$ for $\ell\geq1$\em. 

As $\Omega$ is open in $X$, we have $\cl_{\Omega}(T_\ell\cap\Omega)=\cl(T_\ell)\cap\Omega$. Thus, it is enough to show \em $\cl(T_\ell)\cap\Omega\subset T_\ell$ for each $\ell\geq1$\em. 

Indeed, since $T_\ell$ is closed in $\Omega_\ell$, we have $\cl(T_\ell)\cap\Omega_\ell=T_\ell$. As $\cl(T_\ell)\setminus\Omega_\ell\subset C_k$ for $k\neq\ell$,
$$
\cl(T_\ell)\cap\Omega=\bigcup_{k\geq1}\cl(T_\ell)\cap(\Omega_k\setminus C_k)\subset\cl(T_\ell)\cap\Omega_\ell\cup\bigcup_{k\neq\ell}\cl(T_\ell)\setminus(\cl(T_\ell)\setminus\Omega_\ell)=T_\ell,
$$
as required.
\end{proof}

\begin{lem}\label{proylc}
Let $\pi:Y\to X$ be a proper map with finite fibers between two topological spaces. Let $\{A_i\}_{i\in I}$ be a locally finite family of $Y$. Then $\{\pi(A_i)\}_{i\in I}$ is a locally finite family of $X$.
\end{lem}
\begin{proof}
Let $x\in X$ and write $\pi^{-1}(x):=\{y_1,\ldots,y_r\}$. For each $j=1,\ldots,r$ let $V_j$ be an open neighborhood of $y_j$ that only intersects finitely many $A_i$. Let $C:=Y\setminus\bigcup_{j=1}^rV_j$, which is a closed subset of $Y$. As $\pi$ is proper, $\pi(C)$ is closed. As $\pi^{-1}(x)\cap C=\varnothing$, we have $x\not\in\pi(C)$, so $U:=X\setminus\pi(C)$ is an open neighborhood of $x$. Let us check that \em if $\pi(A_i)\cap U\neq\varnothing$ then there exists $j=1,\ldots,r$ such that $A_i\cap V_j\neq\varnothing$\em. Thus $\{\pi(A_i)\}_{i\in I}$ is a locally finite family of $X$.

Suppose by contradiction that $A_i\cap V_j=\varnothing$ for all $j=1,\ldots,r$. Then $A_i\subset C$, so $\pi(A_i)\subset\pi(C)$ and $\pi(A_i)\cap U\neq\varnothing$, which is a contradiction, as required.
\end{proof}

\end{document}